\documentclass[11pt]{article}
\usepackage{fullpage}%
\usepackage{amsmath}
\usepackage{amsfonts}
\usepackage{amssymb}
\usepackage{pgfplots}
\usepackage{tikz}
\usepackage{tikz-cd}
\usetikzlibrary{intersections}
\usepackage{caption}
\usepackage{subcaption}
\usepackage{hyperref}
\usepackage{pstricks}

\newtheorem{theorem}{Theorem}

\newtheorem{algorithm}[theorem]{Algorithm}

\newtheorem{corollary}[theorem]{Corollary}

\newtheorem{definition}[theorem]{Definition}
\newtheorem{example}[theorem]{Example}

\newtheorem{lemma}[theorem]{Lemma}

\newtheorem{proposition}[theorem]{Proposition}
\newtheorem{remark}[theorem]{Remark}

\newenvironment{proof}[1][Proof]{\noindent\textbf{#1.} }{\ \rule{0.5em}{0.5em}}
\newcommand{\beq}{\begin{equation}}
\newcommand{\eeq}{\end{equation}}
\newcommand\eq{\eqref}

\newcommand\mfm{\mathcal F}
\newcommand\mc{\mathcal M}
\newcommand\emc{\widetilde\mc}

\newcommand\pv{\K[t]^n}
\newcommand\pmat{\K[t]^{n\times n} }
\newcommand\slnk{SL_n(\K) }
\newcommand\bv{\bar \vv }
\newcommand\vv{ {\mathbf v} } 
\newcommand\cv{ V } 
\newcommand\uu{ {\mathbf u} }
\newcommand\bb { {\mathbf b}}
\newcommand\ba { {\mathbf a}}
\newcommand\ww {{\mathbf{w}}}
\newcommand\hh{{\mathbf{h}}}
\newcommand\ch{ H } 
\newcommand\len{\eta}
\newcommand\dg{ \delta }
\newcommand{\gva}{\ensuremath{{G}}} 
\newcommand{\g}{\ensuremath{{g}}} 
\newcommand{\zva}{\ensuremath{\mathcal{Z}}} 
\newcommand{\yva}{\ensuremath{\mathcal{Y}}} 
\def\pz{{z} }
\def\bpz{\overline{\pz} }
\newcommand{\act}{\alpha}

\def\syz{\operatorname{syz}}
\def\rank{\operatorname{rank}}
\def\bez{B\'{e}zout } 
\def\mub{$\mu$-basis } 
\def\QS{Quillen-Suslin }
\def\K{\mathbb K}
\def\R{\mathbb R}
\def\Q{\mathbb Q}
\def\KK{\overline{\K}}
\def\VV{\overline{V}}
\def\l{L}  
\def\cc{\mathfrak c}
\def\rc{\mathcal C}
\def\rv{\mathcal V}
\def\cs{\mathcal K}
\def\orb{\mathcal O}
\def\pW{\mathbf W}
\def\pF{\mathbf F}
\def\pG{\mathbf G}
\def\pM{\mathbf M}
\def\pQ{\mathbf Q}
\def\pU{\mathbf U}
\def\pP{\mathbf P}
\def\qq{$\Box$}
\begin{document}

\title{Equi-affine minimal-degree moving frames \\for polynomial curves}
\author{Hoon Hong\\North Carolina State University \\hong@ncsu.edu\\\url{https://hong.math.ncsu.edu/} \and  Irina A. Kogan\\North Carolina State University\\iakogan@ncsu.edu\\\url{https://iakogan.math.ncsu.edu/}}

\maketitle
\abstract{We develop a theory and an algorithm for constructing minimal-degree polynomial  moving frames for polynomial curves in an affine space. The algorithm is equivariant under  volume-preserving affine transformations of the ambient space and the parameter shifts. We show that any matrix-completion algorithm can be turned into an equivariant  moving frame algorithm via an equivariantization procedure that we develop.  
We prove that if a matrix-completion algorithm is of minimal degree then so is the resulting equivariant moving frame algorithm.   We propose a novel minimal-degree matrix-completion algorithm, complementing the existing body of literature on this topic.}

\tableofcontents
\section{Introduction}
 In  this paper, we use the term  \emph{moving frame} in its original meaning, namely   a movable frame of reference (a coordinate system), which is better adapted to the geometric or physical nature  of a problem than a fixed coordinate system.
 This notion played a central role in the study of smooth curves and surfaces in differential geometry.
  
 The Frenet frame, consisting of the unit tangent and  unit normal assigned at each point  along a planar curve, is, perhaps, the most prominently known example.  It  has \emph{two crucial geometric properties}. The \emph{first property} is  equivariance under rotations and translations in the ambient plane illustrated  in Figure~\ref{fig:eucl.mf}.  When a composition of a translation and a rotation  sends a curve $\Gamma$  to $\tilde \Gamma$ and a  point $p\in \Gamma$ to $\tilde p\in \tilde \Gamma$,  the   Frenet frame at $p$ along $\Gamma$ and the   Frenet frame at $\tilde p$  along $\tilde\Gamma$ are related by the same transformation. Compositions of rotations and translations in the plane comprise a group, called the special Euclidean group\footnote{When reflections are included, the group is called \emph{Euclidean} and is denoted $E_2$.} and denote $SE_2$. The \emph{second property} is equivariance under  orientation preserving  reparameterizations of the curve.
 Assume $\Gamma$ has a smooth parametrization  
${\cc(t)}=\begin{bmatrix}x(t)&y(t)\end{bmatrix}$,  then the Frenet frame along $\Gamma$ has parametrization
\beq\label{eq-TN}T(t)=\begin{bmatrix} \frac{x'(t)}{\sqrt{x'(t)^2+y'(t)^2}}\\ \frac{y'(t)}{\sqrt{x'(t)^2+y'(t)^2}}\end{bmatrix}\text{ and } N(t)=\begin{bmatrix} - \frac{y'(t)}{\sqrt{x'(t)^2+y'(t)^2}} \\\frac{x'(t)}{\sqrt{x'(t)^2+y'(t)^2}}\end{bmatrix}\eeq
Let $\phi\colon\R\to\R $ be a smooth function such that $\phi'(t)>0$ for all $t\in\R$. It is straightforward to check that if we apply  formulas \eqref{eq-TN} to parametrization $\hat{\cc}(t)={\cc}(\phi(t))$ of $\Gamma$ we obtain the Frenet frame with parametrization $T(\phi(t))$ and $N(\phi(t))$.
\begin{figure}[t]
\centering
\begin{minipage}{.40\linewidth}
  \includegraphics[width=1\linewidth]{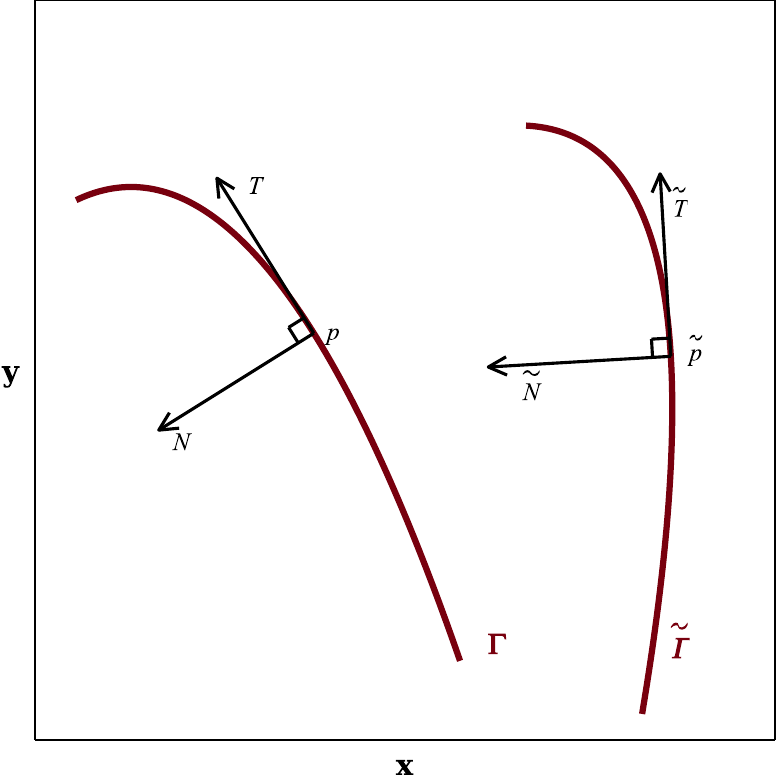}
  \caption{$SE_2$-equivariance of the Frenet frame.}
  \label{fig:eucl.mf}
\end{minipage}
\hspace{.05\linewidth}
\begin{minipage}{.40\linewidth}
  \centering
  \includegraphics[width=1\linewidth]{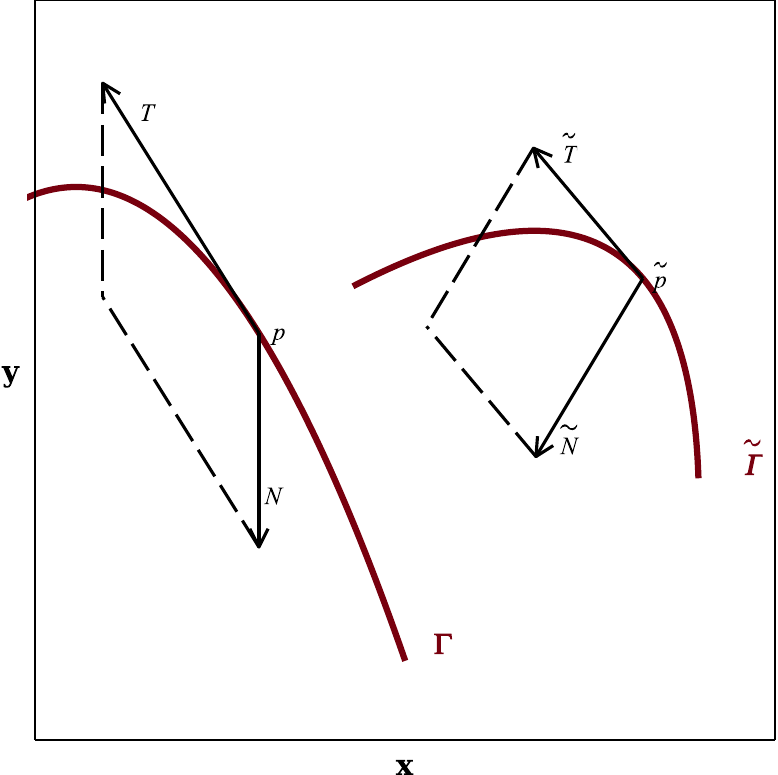}
  \caption{$SA_2$-equivariance of the classical affine frame.}
  \label{fig:aff.mf}
\end{minipage}
\end{figure}

 Frenet's construction generalizes to  curves in $\R^n$, as well as  to higher dimensional submanifolds in  $\R^n$,  with Darboux's frames for surfaces being another classical example.  Frames that are equivariant under the special Euclidean transformations   allow us to  study properties of the geometric shapes independently of their parametrization and position in the space, and to establish their congruence  under rigid motions.  For  the space curves,  in addition to Frenet  frames, other $SE_3$-equivariant frames were considered, in particular,  for applications in computer-aided geometric design and computer graphics \cite{bishop1975, klok1986, gugg1989, farouki2010}.

Following Felix Klein's Erlangen program of 1872 (\cite{klein-germ,
 klein-eng}), where he proposed  to view geometry
 as the study of invariants  under various  group actions on geometric objects, the discovery of frames equivariant under actions of  non-Euclidean groups flourished, culminating  in \'Elie Cartan's general theory of moving frames \cite{C35}. For  its contemporary formulations and generalizations see, for instance, \cite{ivey-landsberg-2016, olver-150}. 
 
 The special affine group $SA_n(\R)$, consisting of volume preserving linear transformations and translations in $\R^n$, is one of the groups of interest with applications including computer vision and shape analysis \cite{faugeras94,COSTH-1998,TOT2020}, human movement and handwriting recognition \cite{GMW-10, FH-2007, MF-2014}. 
 Since the special affine  action preserves neither  the lengths of vectors  nor the angles between them, the $SE_n$-equivariant frames, consisting of unit-length orthogonal vectors,  are not $SA_n$-equivariant.  In Figure~\ref{fig:aff.mf}, we show  the classical  $SA_2$-equivariant frame along a planar curve. It  consists of  two vectors, called the affine tangent and the affine normal, respectfully, such  that the parallelogram defined by these vectors has a unit area. For the explicit formulas  see, for instance,~\cite[p.150]{Gug63}.

 When applied  to polynomial curves, the classical   differential-geometric frames  have  a serious drawback because their expressions are neither polynomial nor even rational. A large body of work has been published to address this issue.  One  possible approach  is to find classes  of curves admitting rational frames. This lead to the  extensive study of the class of polynomial curves $\cc(t)$ with polynomial speed $|\cc(t)'|$, called   Pythagorean hodograph\footnote{For any $n\geq 2$, Pythagorean hodograph condition is necessary   for rationality for Frenet frame, but only for $n=2$ it is sufficient.} curves, and its subclasses \cite{farouki2008, cheng2016,  farouki2012, farouki2017}, mostly for planar and space curves in Euclidean geometry. As it can be seen from these works, the  set  of polynomial curves admitting $SE_n$- and  reparameterization-equivariant  rational frames  is very small.  Another approach is to find appropriate  rational approximations of non-rational  equivariant frames \cite{maurer1999, farouki2003, farouki2022}.

 In this paper, we address the problem of constructing \emph{polynomial} equi-affine (i.e $SA_n$-equivariant)  frames on the set of \emph{generic polynomial curves\footnote{Precise conditions for a curve being generic are given in Definition~\ref{def-regc}.}. However,  in order to remain in  the realm of the \emph{polynomial algebra}, we need to  restrict the allowed set of re-parameterizations to  parameter-shifts.} More precisely,  we are proposing an algorithm that assigns a   \emph{polynomial} moving  frame $\ww_1(t),\dots,\ww_n(t)$  along a generic polynomial curve $\cc(t)$ in  $\K^n$,  where $n>1$ and $\K$ is a field of characteristic zero. The algorithm combines the following important algebraic and geometric properties:
\begin{enumerate}
\item For all $t$,  $\ww_1(t)$ is tangent to the curve at the point $\cc(t)$.
\item For all $t$, the frame has unit volume: $| \begin{matrix}\ww_1(t)&\dots&\ww_n(t)\end{matrix}|=1$
\item \label{mdeg} The frame is of minimal degree per Definition~\ref{def-mdeg} below. 

\item \label{san} If the input is affected by  an $SA_n(\K)$  transformation in the ambient space $\K^n$, then  so does the output.
\item  \label{sa1} If the input is affected by  a parameterization shift, then  so does the output. 
 \end{enumerate}
It is important to note that  properties ~\ref{san}--\ref{sa1} are not the properties of the algorithm's output, but rather the properties  of the algorithm itself. In other words, if we view the algorithm  as the  map that assigns a frame to the input  polynomial curve, properties  ~\ref{san}--\ref{sa1} say that  this map is equivariant under affine transformations  of the ambient space and parameter-shifts, as defined  rigorously in Definition~\ref{def-eamf} and discussed in Remark~\ref{rem-equiv}. Thus, a more rigorous title of the paper would be ``Equi-affine minimal-degree moving frame \emph{maps} for polynomial curves''. Following the traditional differential geometry terminology we omitted the word ``map'' from the title.

 Our \emph{main result},  summarized in Theorem~\ref{thm-mdeamf},  shows how an equi-affine minimal-degree moving frame algorithm can be built on top  of \emph{any} minimal-degree matrix completion algorithm, via an explicit  minimal-degree preserving ``equivariantization'' procedure we propose.  The matrix-completion problem is a well studied problem in computational algebra, often in relation to the  Quillen-Suslin problem\footnote{As discussed in Remark~\ref{rem-qs}, a minimal-degree \QS matrix does not correspond to a \emph{minimal-degree} matrix completion.} (see, for instance, \cite{zhou-2012, zhou-labahn-2014, fitchas-galligo-1990, logar-sturmfels-1992, caniglia-1993,park-woodburn-1995,lombardi-yengui-2005,fabianska-quadrat-2007,omf-2017,hough-2018}).   Our equivariantization process  was inspired by the  general ``invariantization'' process\footnote{As discussed in Remark~\ref{rem-FO}, Fels-Olver's construction belongs to the smooth category and starts with a choice of a local cross-section to the orbits of a Lie group acting on a smooth manifold. The cross-section \emph{implicitly} defines a locally equivariant smooth map from an open subset of the smooth manifold to the Lie group. To adopt the method to our problem, we need to drop not only the smoothness, but even the continuity assumption, replacing a cross-section with a   set (with no additional structure) of canonical representatives of each orbit. We \emph{explicitly} build a global (non-continuous !) equivariant map from the set of regular polynomial vectors to the relevant group.} introduced by Fels and Olver in \cite{FO99}, but its adaptation  to the problem at hand turned out to be nontrivial. One of the major challenges was achieving \emph{simultaneous} equivariantization under affine transformations in the ambient space and under parameter shifts.
 
Our \emph{secondary  result}, summarized in Theorem~\ref{thm-mc},  shows that one can obtain a minimal degree matrix completion of  a given vector
 $\vv$ by  adjoining $\vv$ with a $\mu$-basis of  a minimal-degree \bez vector of $\vv$. Since algorithms  for computing  $\mu$-bases\footnote{Also known as \emph{optimal-degree kernels} and \emph{minimal bases}.}  and \bez vectors are abundant  (see, for instance, \cite{sederberg-chen-1995,cox-sederberg-chen-1998, zheng-sederberg-2001, chen-wang-2002, chen-cox-liu-2005, song-goldman-2009, jia-goldman-2009,tesemma-wang-2014, hhk2017, jia-shi-chen-2018, beelen1987,antoniou2005,zhou-2012,EGB}), Theorem~\ref{thm-mc}  leads to a novel, easy-to-implement algorithm for computing a minimal degree matrix completion, supplementing the existing body of literature on this topic.

The structure of the paper unwraps  our construction  starting with  the outermost layer.
In Section~\ref{sect-prob}, we rigorously state the problem after introducing all necessary definitions. We start with a definition of the degree of a polynomial vector  or  a polynomial matrix. We proceed with defining the set $\rc$ of generic polynomial curves, which serves as a maximal possible domain for an equi-affine minimal-degree moving frame map. We recall the standard definitions  of a  group action and  related notions,  including invariance and equivariance.   This allows us to give a definition of equi-affine minimal-degree moving frame map and rigorously state the problem.

In Section~\ref{sect-mf}, we show how an equivariant moving frame algorithm can be built on top of any matrix-completion algorithm. This is done in two stages: In Section~\ref{subs-eamf}, we  show that an equivariant moving frame algorithm can be built on top of any equivariant matrix-completion algorithm. In Section~\ref{subs-eqmc}, we show that an equivariant matrix-completion algorithm  can be built on top of any matrix-completion algorithm.  To do so, we develop a general ``equivariantization'' process  (described in Section~\ref{subs-equiv}),
which was  inspired by the  invariantization process introduced in the fundamental work \cite{FO99}.   Application of this process to the problem at hand  requires, however, some ingenuity, as reflected in Theorem~\ref{thm-rho12}.
  In this section, we are not yet concerned with the degree-minimality.  

Section~\ref{sect-mmc} is devoted to the matrix-completion problem. In Section~\ref{subs-mmc}, we show that by adjoining $\vv\in\pv$, such that $\gcd(\vv)=1$, with  a $\mu$-basis of a \bez vector $\bb$ of $\vv$, one produces a matrix completion of $\vv$ of degree equal to $\deg\vv+\deg\bb$. If $\bb$ is of minimal degree then so is the matrix completion.    
There is a variety of existing algorithms and methods to compute a minimal-degree \bez vector and a $\mu$-basis. 
After showing, in Section~\ref{subs-mind},  that the equivariantization process presented in Section~\ref{subs-eqmc}    preserves the minimality of the degree, we have all the needed ingredients to build an equi-affine minimal-degree moving frame algorithm.

In Section~\ref{sect-mdeamf},  we pull together all the preceding results to formulate the main theorem of our paper, Theorem~\ref{thm-mdeamf}, and the  corresponding equi-affine minimal-degree moving frame algorithm.
The Maple implementation of the algorithm can be found at 
\begin{center}
\url{https://github.com/Equivariant-Moving-Frame/EAMFM}.
\end{center}
 
In Appendix, we briefly describe an approach for computing a minimal-degree \bez vector and a $\mu$-basis, which we used in the above Maple implementation.  This approach,   first developed in \cite{hhk2017, omf-2017, hough-2018},  reduces both problems to the row-echelon reduction of a matrix over the field $\K$.

\section{Problem statement and definitions}\label{sect-prob}
Throughout the paper,  $\K$ denotes a field of characteristic zero  and $\KK$ its algebraic closure, while $\K[t]$ denotes the ring of polynomials over a field $\K$, $\pv$ denotes the set of polynomial vectors of length $n$ and $\K[t]^{n\times l}$ denotes the set of $n\times l$ polynomial matrices.
In examples, its is natural to assume $\K=\mathbb{Q}$ for computations and  $\K=\mathbb{R}$ for visualization. From now on, we will use  {\bf bold  faced} font to denote polynomials, as well as polynomial  vectors and matrices.

Our goal is to design an algorithm whose input is a polynomial curve in an  $n$-dimensional affine space   and whose output is an $n$-tuple of linearly independent polynomial vectors, such that the first vector is tangent to the input  curve, and  the  parallelepiped spanned by all $n$ vectors has constant volume equal to one. The output can be visualized  as \emph{a frame} of basis vectors  in  $n$-dimensional affine space \emph{moving} along the curve. In addition, we impose  two desirable geometric and algebraic properties on our algorithm:  degree minimality of the output and  equivariance of the algorithm under  affine transformations of the ambient space and  parameterization-shifts. A rigorous statement of the problem is given at the end of this section and is based on the following definitions.  
\begin{definition}[degree]\label{def-mdeg}  
The \emph{degree} of a polynomial  vector $\vv=\begin{bmatrix}\vv_1\\ \vdots\\ \vv_n\end{bmatrix} \in \K[t]^n$ is defined by
\[\deg\vv=\max_{i=1}^{n}\ \deg \vv_{i}.\]
The \emph{degree} of a polynomial matrix ${\pW}\in \K\left[t\right]^{n\times\ell}$ 
is defined by 
\[
\deg {\pW}=\sum_{j=1}^{\ell} \deg \pW_{*j} \;\;\;\;\;(\text{where $\pW_{*j}$ denotes the $j$-th column vector of $\pW$})
\]
The set of polynomial vectors of length $n$ and degree $d$ is denoted by $\pv_d$. 
\end{definition}

\begin{definition}[generic polynomial curves]\label{def-regc}  A column vector $\cc\in \K\left[  t\right]  ^{n}$ is called a
\emph{generic polynomial curve} if  
\begin{enumerate}
\item For every $t\in \overline{\K}\ \ $we have $\cc^{\prime}\left(  t\right)  \neq 0$.
\item The curve does not lie in a proper affine subspace of $\K^{n}$.
\item $\deg \cc>n$.
\end{enumerate}
The set of all generic polynomial curves is denoted by $\mathcal{C}$.
\end{definition}

\begin{remark}\rm  
The three conditions in the definition are motivated by the following considerations.
\begin{enumerate}
\item The requirement that the tangent vector is non-zero is the standard genericity condition in differential geometry.  If $\K$ is not algebraically closed, our requirement is stronger  and  turns  out to be necessary and sufficient for a moving frame  to exist along the curve.

\item If a curve lies in a proper affine subspace of $\K^{n}$, then after an appropriate affine transformation it can be represented by a polynomial vector  in a lower dimensional linear space. The first and the second conditions combined are necessary and sufficient for a  curve to be included in the domain a moving frame algorithm, which is equivariant under the ambient space affine transformations, however, it is not sufficient for a curve to be included in the domain of a  parameter-shift and ambient space equivariant moving  frame algorithm (see Definition~\ref{def-eamf})

\item For the second condition to hold it is necessary that  $\deg\cc\geq n$. The sharp inequality  in the third condition, in combination with the above two conditions,  turns out to be   necessary and sufficient   for a curve to be included in the domain of a fully equivariant moving frame algorithm.
\end{enumerate}
When $d>n$, the set of generic curves of degree $d$  is Zariski open and dense in  $\pv_d$.
\qq
\end{remark}

 \begin{example}\rm The  polynomial vector
 $\cc=\left[
\begin{array}
[c]{c}%
t\\
t^{2}\\
t^{4}+1
\end{array}
\right]$ defines  a generic polynomial curve of degree~4 in $\Q^3$ because:
\begin{enumerate}
\item For every $t\in \overline{\Q}=\mathbb{C}\ \ $we have $\cc^{\prime}\left(  t\right)=\left[
\begin{array}
[c]{c}%
1\\
2t\\
3t^{3}%
\end{array}
\right]   \neq0$.

\item  The curve
$\cc=\left[
\begin{array}
[c]{c}%
0\\
0\\
1
\end{array}
\right]  +\left[
\begin{array}
[c]{c}%
t\\
t^{2}\\
t^{4}%
\end{array}
\right]  =\left[
\begin{array}
[c]{c}%
0\\
0\\
1
\end{array}
\right]  +\underset{Q}{\underbrace{\left[
\begin{array}
[c]{cccc}%
 1 & 0 & 0 & 0\\
 0 & 1 & 0 & 0\\
 0 & 0 & 0 & 1
\end{array}
\right]  }}\left[
\begin{array}
[c]{c}%
t\\
t^{2}\\
t^{3}\\
t^{4}%
\end{array}
\right]  $
does not lie in a proper affine subspace of $\K^3$ because $\rank Q=3=n$. 
\item  $\deg \cc=4>n=3$.
\end{enumerate}
\qq
\end{example}

\begin{definition}[moving frame]\label{def-mf}
We say $\pF\in \K\left[  t\right]  ^{n\times
n}$ is a \emph{moving frame} of $\cc\in  \mathcal{C}$ if

\begin{enumerate}
\item $\pF=\left[
\begin{array}
[c]{cc}%
\cc^{\prime} & \cdots
\end{array}
\right]$
\item $\left\vert \pF\right\vert =1$.
\end{enumerate}
A moving frame $\pF$  of $\cc$ is said to be of \emph{minimal degree} if for any moving frame $\pG$ of $\cc$
$$\deg \pF\leq\deg \pG.$$
A map $\mfm\colon\mathcal{C}\longrightarrow\K\left[  t\right]
^{n\times n}\ $ is called a \emph{moving frame map} if $\mfm\left(
\cc\right)  $ is a moving frame of $\cc$ for every $\cc\in\mathcal{C}$.  The map $\mfm$ is of \emph{minimal degree} if $\mfm\left(
\cc\right)  $ is a minimal degree moving frame of $\cc$ for every $\cc\in\mathcal{C}$.
\end{definition}

The first condition in Definition~\ref{def-mf} requires the first vector in the frame to be the tangent  vector of the curve. The second condition  guarantees that for all $t\in\K$, the columns of $\pF(t)$ comprise a set of $n$ linearly independent  vectors in $\K^n$ and so can be visualized as a basis (or a frame) moving along the curve. Moreover, the parallelepiped defined by this frame has a constant volume 1.  From the algebraic perspective, one can say that a moving frame of $\cc$ is a \emph{matrix completion} of its velocity vector $\cc'$ (see Definition~\ref{def-mc} in Section~\ref{subs-eamf}).
 \begin{example}\rm Consider 
 $\cc=\left[
\begin{array}
[c]{c}%
t\\
t^{2}\\
t^{4}+1
\end{array}
\right] \in \mathcal{C}$. Then it is straightforward to verify that 
matrices
\[
\pF=\left[
\begin{array}
[c]{ccc}%
1 & 0 & 0\\
2t & 2 & 0\\
4t^{3} & 0 & \frac{1}{2}
\end{array}
\right] \text{ and }   
\pG=\left[
\begin{array}
[c]{ccc}%
1 & 0 & 0\\
2t & 2 & 0\\
4t^{3} & t^{2023} & \frac{1}{2}%
\end{array}
\right]
\]
are  moving frames of $\cc$.
Note that%
\begin{align*}
\deg \pF  &  =\max\left[
\begin{array}
[c]{c}%
0\\
1\\
3
\end{array}
\right]  +\max\left[
\begin{array}
[c]{c}%
-\infty\\
0\\
0
\end{array}
\right]  +\max\left[
\begin{array}
[c]{c}%
-\infty\\
-\infty\\
0
\end{array}
\right]  =3+0+0=3,\\
\deg \pG  &  =\max\left[
\begin{array}
[c]{c}%
0\\
1\\
3
\end{array}
\right]  +\max\left[
\begin{array}
[c]{c}%
-\infty\\
0\\
2023
\end{array}
\right]  +\max\left[
\begin{array}
[c]{c}%
-\infty\\
-\infty\\
0
\end{array}
\right]  =3+2023+0=2026.
\end{align*}
Therefore $\pG$ is not of minimal degree. On the other hand, it is clear that the degree of a moving frame of $\cc$ is at least  $\deg \cc'=3$, and so $\pF$ is of minimal degree.

Any moving frame algorithm, including the one  introduced in Section~\ref{sect-mdeamf}, is an example of a moving frame map. 
\qq
\end{example}
Before giving the definition of an equi-affine moving frame map, we review the standard definitions of a group action, as well as definitions   of invariant and  equivariant maps. 


\begin{definition}[action]\label{def-act}
An  \emph{action}  of a group $\gva$  on a set $\zva$ is a map $\act\colon  \gva \times \zva \to \zva $  satisfying the following two properties:
\begin{enumerate}
    \item[i.] \label{ai} \emph{Associativity:}  $ \act(\g_1 \star  \g_2, \pz) = \act(\g_1, \act (\g_2, \pz ))$, $\forall \g_1, \g_2\in \gva$  and  $\forall \pz\in \zva$.
\item[ii.]\label{aii} \emph{Action of the identity element:}     $\act(e,\pz) = \pz$, $\forall  \pz \in \zva$.  
\end{enumerate}
We say that an action $\act$ is \emph{free} if  for all $\pz\in\zva$ and $\g\in \gva$:
$$\act( \g, \pz )=\pz\qquad\Longleftrightarrow\qquad \g=e.$$
\end{definition}

\begin{remark} \label{rem-act}  \rm\
        
\begin{enumerate}\item In the above definition, $\star$ denotes the  group operation. The abbreviation  $\act(\g,\pz)=\g\cdot \pz$ will be used  when the map $\alpha$ is clear from the context. 
\item To avoid cluttering of parentheses, we \emph{will sometimes use
$g_1 \cdot g_2 \cdot \pz$ to denote either of the two  quantities:  $g_1 \cdot(g_2 \cdot \pz)$ and $(g_1 \star g_2) \cdot \pz$, whose equality is due to the associativity of the action.}
\item If two groups  $\gva_1$ and $\gva_2$ act on $\zva$ and  their actions commute, then there is a well defined action of the direct product $\gva_1 \times \gva_2$, where for $(\g_1,\g_2)\in \gva_1 \times \gva_2$, $\pz\in \zva$:
$$(\g_1,\g_2)\cdot\pz=\g_1\cdot (\g_2\cdot\pz)=\g_2\cdot (\g_1\cdot\pz).$$
Conversely, an action of a direct product $\gva_1 \times \gva_2$ on $\zva$, defines commuting actions of $\gva_1$ and $\gva_2$ on $\zva$  by
$$ g_1\cdot \pz=(g_1,e)\cdot\pz\text{ and  } g_2\cdot \pz=(e, g_2)\cdot\pz$$
\item For a fixed $\g\in \gva$ a group action $\act\colon  \gva \times \zva \to \zva $ induces a \emph{bijection} $\act_\g\colon \zva\to \zva$:
$$\act_\g(\pz)=\act(\g,\pz),$$
whose inverse is   $\act_{\g^{-1}}$.
\qq
\end{enumerate}
\end{remark}
\begin{example} \rm  The following actions play an important role in this paper.

\begin{enumerate}
\item The group operation induces an action  $\act\colon  \gva \times \gva \to \gva $ of the group on itself:
\beq\label{self-act}\g_1\cdot \g_2=\g_1\star \g_2.\eeq
In this paper, we will see 
\begin{enumerate}
\item An action of the field  $\K$ on itself    with  the group operation being the addition.
\item  An action of the special\footnote{In the context of matrix groups ``special'' means of determinant one.} linear group $$SL_n(\K)=\{L\in\K^{n\times n}|\,|L|=1\}$$ on itself with the group operation being the multiplication of matrices. 
\end{enumerate}
\item An action of   $SL_n(\K)$ on the set $\K[t]^{n\times l}$ of  $n\times l$ polynomial matrices (including the set of polynomial vectors $\K[t]^{n}$ when $l=1$) by   matrix multiplication:  
\beq\label{L-act} \l\cdot \pW=\l \pW\text{ for } L\in  SL_n(\K) \text{ and }\pW\in \K[t]^{n\times l}.\eeq

\item An action of  $\K$, with the group operation being the  addition, on the set $\K[t]^{n\times l}$ by the parameter shift:
\beq\label{s-act} s\cdot \pW(t)=\pW(t+s)\text{ for } s \in \K \text{ and }\pW\in \K[t]^{n\times l}.\eeq

\item Since actions \eq{L-act} and \eq{s-act} on $\K[t]^{n\times l}$ commute, in combination, they define  an action of the direct product $\slnk\times\K $ on  $\K[t]^{n\times l}$.
\beq\label{Ls-act} (L,s)\cdot \pW(t)=L\pW(t+s)\text{ for }  L\in  SL_n(\K) ,s \in \K \text{ and }\pW\in \K[t]^{n\times l}.\eeq
\item An action of the group  $\K^n$ of constant vectors, with the group operation being the addition of vectors, on $\K[t]^{n}$ by translation:
\beq\label{a-act} a\cdot \vv=\vv+a\text{ for } a \in \K^n \text{ and }\vv\in \K[t]^{n}.\eeq
\item Actions  \eq{L-act} and \eq{a-act} on  $\K[t]^{n}$ do not commute, but, in combination, they define an action of the semi-direct product $SL_n(\K)\ltimes \K^n$  on $\K[t]^{n}$. This semi-direct product is called \emph{the special affine group}t and is denoted $SA_n(\K)$.
\item In combinations actions \eq{L-act}, \eq{s-act} and \eq{a-act} define an action of $ SA_n\times\K$ on $\pv$ by
\beq \label{all-act}(L,a, s)\cdot \vv(t)=L\vv(t+s)+a \text{ for }  L\in SL_n(\K), a \in \K^n, s\in\K \text{ and }\vv\in \K[t]^{n}.\eeq
\end{enumerate}
\qq
\end{example}
We note the following simple but necessary result. 
\begin{proposition} Action \eq{all-act} on $\pv$ restricts to $\rc$:
$$ (L,a, s)\cdot \cc\in \rc \text{ for  any }  (L,a,s)\in SA_n\times\K \text{ and }\cc\in \rc.$$
\end{proposition}

\begin{definition}[orbit]\label{def-orb}  For an action $\act\colon  \gva \times \zva \to \zva $, the \emph{orbit}  of an element $\pz\in\zva$ is the set 
\beq\label{eq-orb} \orb_\pz=\left\{\act(\g,\pz)\,|\,\g\in\gva\right\}. \eeq
\end{definition} 
If the action is free then for each $\pz\in \zva$ the map $\act_\pz\colon \gva\to\orb_\pz$, defined by $\act_\pz(\g)=\act(\g,\pz)$, is a bijection.
\begin{figure}[h]
  \centering
  \begin{subfigure}[b]{0.45\textwidth}
  \centering
  \begin{tikzcd}[column sep=normal]
    \zva \arrow[swap]{rd}{f} \arrow{rr}{\act_\g}&          & \zva \arrow{ld}{f}\\
                                                &\mathcal Y&
  \end{tikzcd}
  \caption{Invariance.}
  \label{fig:inv}
  \end{subfigure}
  \hfill
  \begin{subfigure}[b]{0.45\textwidth}
  \centering
  \begin{tikzcd}[column sep=large]
    \zva       \arrow{r}{\act_\g  }   \arrow[swap]{d}{f}      & \zva  \arrow{d}{f}\\
    \mathcal Y \arrow{r}{\beta_\g }                           & \mathcal Y 
  \end{tikzcd}
  \caption{Equivariance.}
  \label{fig:eqv}
  \end{subfigure}
\caption{Commutative diagrams.}
\end{figure}
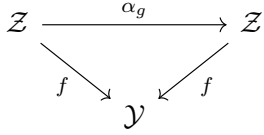
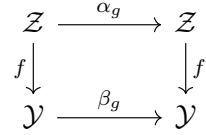
\begin{definition}[invariance]\label{def-inv} A map $f\colon\zva\to\yva$ is called \emph{invariant} under the action~$\act\colon  \gva \times \zva \to \zva $ if 
\beq\label{invf}f\left(\alpha (\g,\pz)\right)=f(\pz), \qquad \forall (\g,\pz)\in \gva\times\zva. \eeq
\end{definition}
Equivalently, we can say that the diagram in Figure~\ref{fig:inv} commutes for all $\g\in\gva$ or that $f$ is constant along each orbit.
\begin{definition}[equivariance]\label{def-eqv} A map $f\colon\zva\to\yva$ is called \emph{equivariant} under the actions $\act\colon  \gva \times \zva \to \zva $ and $\beta\colon  \gva \times \yva \to \yva $ if 
\beq\label{eqvf}f\left(\alpha (\g,\pz)\right)=\beta(\g,f(\pz)), \qquad \forall (\g,\pz)\in \gva\times\zva. \eeq
\end{definition}
Equivalently, we can say that the diagram in Figure~\ref{fig:eqv} commutes for all $\g\in\gva$ or that $f$ maps an orbit of action  $\act$ to an orbit of action $\beta$.

\noindent Now we have all the necessary notions and notations to give the central definition of our paper.
\begin{definition}[equi-affine moving frame map]\label{def-eamf} A moving frame map $\mfm\colon {\mathcal C} \to \K[t]^{n\times n}$ is called \emph{equi-affine} if

\begin{enumerate}
\item $\underset{\cc\in\mathcal{C}}{\forall}\ \underset{\l\in SL_{n}\left(
\K\right)  ,\ a\in \K^{n}}{\forall}\mathcal{F}\left(  a\cdot (\l\cdot
\cc)\right)  =\l\cdot\mathcal{F}\left(  \cc\right)  $\hfill (ambient equivariance),

\item $\underset{\cc\in\mathcal{C}}{\forall}\ \underset{s\in \K}{\forall
}\ \ \ \ \ \ \mathcal{F}\left(  s\cdot \cc\right)  =s\cdot\mathcal{F}\left(
\cc\right)  $\hfill (parameter-shift equivariance),
\end{enumerate}
where $L \cdot $, $a \cdot $, and $s\cdot $ are defined by \eqref{L-act}, \eqref{a-act} and  \eqref{s-act}, respectively.
\end{definition}
\begin{remark}\label{rem-equiv}\rm\hfill
\begin{enumerate}
\item The first condition requires $\mfm$ to be equivariant under the $SL_n(\K)$-action on $\mathcal{C}$ and on~$\K[t]^{n\times n}$ given by \eqref{L-act}, and to be invariant under the action \eqref{a-act} on  $\mathcal{C}$. Compositions of transformations \eqref{L-act} and \eqref{a-act}  is an action of the special affine group~$SA_n(\K)=SL_n(\K)\ltimes \K^n$ on $\mathcal{C}$. If we assume that $\K^n$ acts trivially on  $\K[t]^{n\times n}$, i.e.~$a\cdot \pW=\pW$ for all~$a\in \K^n$ and~$\pW\in \K[t]^{n\times n}$, then we can say that $\mfm$  is $SA_n$-equivariant. This justifies the term \emph{equi-affine moving frame map}.\footnote{From another perspective, this terminology can be justified by the fact that  the groups $SA_n(\K)$ is often called the \emph{equi-affine group} rather than \emph{the special affine group}.} From the geometric point of view the first condition tells us that $\mfm$ respects the affine transformations of the ambient space $\K^n$ and, therefore, we call this condition \emph{ambient equivariance}.

\item The second condition requires $\mfm$ to be equivariant under the $\K$-actions on $\mathcal{C}$ and on $\K[t]^{n\times n}$ given by \eqref{s-act}. From the geometric point of view, this condition tells us that $\mfm$ respects  parametrization-shifts and, therefore, we call this condition \emph{parameter-shift equivariance}

\item  One can also note that  the $\K$-action on the parameter set, given by $s\cdot t=t+s$,  can be viewed as  $SA_1$-action on $\K^1$,  and thus the two conditions in Definition~\ref{def-eamf}  require $\mfm$ to behave naturally with respect to the special affine actions on the ambient space and on the parameter space. In other words, an equi-affine moving frame map is  $SA_1\times SA_n$-equivariant.
\end{enumerate}
\qq
\end{remark}
We now give a precise formulation of the problem addressed in this paper. 
\vskip2mm
\noindent {\bf Problem:} Design an algorithm $\mfm$  whose input  is a generic polynomial curve $\cc\in \mathcal C$ and whose output   is a  minimal degree moving frame of $\cc$ per Definition~\ref{def-mf}, such that the map $\mfm\colon {\mathcal C} \to \K[t]^{n\times n}$ satisfies  equivariance conditions of     Definition~\ref{def-eamf}.  

\section{Equi-affine  moving frames (\rm{\bf EAMF})}\label{sect-mf}
The goal of this section is to show how  any  matrix-completion algorithm can be turned  into an equi-affine moving frame algorithm. 
 In Section~\ref{subs-eamf}, we prove  that an equivariant moving frame algorithm can be built on top of any equivariant matrix-completion algorithm.  In Section~\ref{subs-equiv}, we describe a general equivariantization process, which, in principle, can be applied to a variety of problems. In Section~\ref{subs-eqmc},  we specialize  the  equivariantization process to the problem at hand to show how an equivariant matrix-completion algorithm  can be built on top of any matrix-completion algorithm.

\subsection{\rm{\bf{EAMF}} from an equivariant matrix-completion (\rm{\bf{EMC}}) }\label{subs-eamf}

Given a vector $\vv(t)=\left[\begin{array}
[c]{c}%
\vv_1(t)\\
\vdots \\
\vv_n(t)
\end{array}
\right]\in  \K[t]^n$, a matrix-completion problem consists of finding a polynomial matrix $\pM\in\K[t]^{n\times n}$, whose first column is $\vv$ and whose determinant equals one. From Definition~\ref{def-mf},  we see that a moving frame of $\cc$ is a matrix-completion of $\vv=\cc'$. The following lemma describes the subset of polynomial vectors that can appear as tangents to generic polynomial curves:

\begin{lemma}[tangents to generic curves]\label{lem-tv}
Assume $\vv=\cc'$ for some generic, per  Definition~\ref{def-regc}, polynomial curve $\cc\in \pv$. Then 
\begin{enumerate}
\item $\gcd(\vv)=1$, where the $\gcd$ of a polynomial vector is defined as \emph{the greatest common divisor of its components  normalized so that the leading coefficient is 1}.
\item The components of $\vv$ are linearly independent over $\K$.
\item $\deg \vv\geq n$.
\end{enumerate}
\end{lemma}
\begin{proof}
\begin{enumerate}
\item The first condition in Definition~\ref{def-regc} requires  that  $\vv(t)  \neq 0$  for every $t\in \overline{\K}$, which is equivalent to  $\gcd(\vv)=1$.
\item Let $\deg\cc=d$, then we can write: 
$$\cc=C \left[\begin{array}
[c]{c}%
1\\
t\\
\vdots \\
t^d
\end{array}\right]=C_{*0}+Q \left[\begin{array}
[c]{c}%

t\\
\vdots \\
t^d
\end{array}\right]
$$
where $C$ is $n\times (d+1)$ matrix of the coefficients of $\cc$, $C_{*0}$ is its first column and $Q$ is an   $n\times d$ matrix comprised of  the  remaining  $d$ columns  of $C$. The curve $$\cc_0=Q \left[\begin{array}
[c]{c}%
t\\
\vdots \\
t^d
\end{array}\right]$$ is a translation of $\cc$, such that $\cc(0)$ is at the origin. Therefore, the second condition in Definition~\ref{def-regc} is equivalent to $\rank Q=n$.

On the other hand,
$$\vv=\cc'=Q \left[\begin{array}
[c]{c}%
1\\
\vdots \\
dt^{d-1}
\end{array}\right]$$ 
 and  so the linear independence of components  of $\vv$ is also equivalent to  $\rank Q=n$.  

\item  The third condition in Definition~\ref{def-regc} is equivalent to  $\deg \vv=\deg \cc-1\geq n$.
\end{enumerate}
\end{proof}

\begin{definition}[regular polynomial vectors] \label{def-regpv}  A polynomial  vector $\vv\in\pv$ is called \emph{regular} if it  satisfies the three conditions of Lemma~\ref{lem-tv}. The set  of such vectors is denoted $\rv$.
\end{definition}
Using Lemma~\ref{lem-tv} and its proof, one can easily derive the following proposition.
\begin{proposition} \label{prop-diff}The differentiation $\cc\to\cc'$ is a well defined  surjective map $\rc\to \rv$ from the space of generic curves to the space of regular polynomial vectors. This map is invariant under the action of $\K^n$ by  translations defined by \eq{a-act}.\end{proposition}
\begin{definition}[matrix-completion]\label{def-mc} For $\vv\in\pv$, we say that $\pM\in\pmat$ is a \emph{matrix-completion} of   $\vv$ if 
\begin{enumerate}
\item $\pM _{*1}=\vv$, where $\pM  _{*1}$ stands for the first
column of the matrix $\pM $,
\item $|\pM|=1$.
\end{enumerate}
A matrix-completion $\pM$ is of \emph{minimal degree} if $\deg \pM\leq\deg {\bf N}$ for any other matrix-completion~${\bf N}$ of $\vv$.
 
 A map $\mc\colon\rv\to\pmat$ is called a \emph{matrix-completion map} if $\mc( \vv)$ is a matrix-completion of $\vv$ for every $\vv\in\rv$.
 The map $\mc$ is of \emph{minimal degree} if  $\mc( \vv)$  is a minimal degree  matrix-completion of $\vv$ for every $\vv\in\rv$.
\end{definition}
\begin{example}\label{ex-mc} \rm Let  $$\vv=\begin{bmatrix} t^6+1\\t^3\\t\end{bmatrix}.$$ Each of following three matrices is a matrix completion of $\vv$
\beq\label{3mc}\pM_1=\begin{bmatrix} t^6+1&0 &t^3\\t^3&0&1\\t&-1&0\end{bmatrix},\quad 
\pM_2=\begin{bmatrix} t^6+1&1 &-t^2\\t^3&1&0\\t&t&1\end{bmatrix},\quad 
\pM_3=\begin{bmatrix} t^6+1&1 &t^3\\t^3&0&1\\t&t-1&0\end{bmatrix}. \eeq 
We have 
$$\deg\pM_1=\deg\pM_2=9\text{ and }\deg\pM_3=10.$$
As we will show in Example~\ref{ex-mmc}, $\pM_1$ and $\pM_2$ have the minimal degree. Nonetheless, column-wise their degree are different  (except, of course, for the first column, which is  equal to $\vv$ for any matrix completion of $\vv$).  In Section~\ref{sect-mmc} we will explicitly construct an equivariant minimal-degree matrix-completion map.
\qq
\end{example}
Before defining an equivariant matrix-completion map, we observe 
that the $SL_n(\K)\times\K$-action~\eq{Ls-act} can be restricted to the set of regular vectors $\rv$. We  also need the following lemma.
\begin{lemma} \label{lem-mc}Let $\mc\colon \rv\to  \K[t]^{n\times n}$  be a  matrix-completion map.  For any  $(L,s)\in \slnk\times \K$, the map $\emc\colon \rv\rightarrow \K[t]^{n\times n}$ 
defined by 
\beq\label{eq-act-mc}\emc(\vv)=(L,s)\cdot \mc\left((L,s)^{-1}\cdot\vv\right)\eeq
 is  a  matrix-completion map.
\end{lemma}
\begin{proof} We check the two properties in Definition~\ref{def-mc}:
\begin{enumerate}
\item \begin{align*}\emc(\vv)_{*1}&=\big[(L,s)\cdot \mc\left((L,s)^{-1}\cdot\vv\right)\big]_{*1}=(L,s)\cdot \big[\mc\left((L^{-1},-s)\cdot\vv\right)\big]_{*1}\\
&=(L,s)\cdot \left((L^{-1},-s)\cdot\vv\right)=\vv.
\end{align*}
\item  \begin{align*} \big|\emc(\vv)\big|&=\big|(L,s)\cdot \mc\left((L,s)^{-1}\cdot\vv\right)\big|=s\cdot \big|L\mc\left((L,s)^{-1}\cdot\vv\right)\big|\\&=s\cdot \big|L\big|\big|\mc\left((L,s)^{-1}\cdot\vv\right)\big|=1.\end{align*} 
\end{enumerate}
\end{proof}

We are now ready to give a formal definition  of a  matrix-completion map which is equivariant under linear transformations of the ambient space and the parameter shift. 

\begin{definition}[equivariant matrix-completion map]\label{def-emc} 
A matrix-completion map $\mc$  is called \emph{$SL_n(\K)\times\K$-equivariant} if  it is equivariant under the action \eqref{Ls-act} on   $\rv$ and $ \K[t]^{n\times n}$, i.e.
\beq\label{sL-mc}
\underset{\vv\in\rv}{\forall}\ \underset{(\l,s)\in SL_{n}\left(\K\right)\times \K }{\forall}\ \mc\left( (\l,s)\cdot \vv\right)  
=(\l,s)\cdot\mc\left(  \vv\right).  
\eeq
\end{definition}
From Proposition~\ref{prop-diff}, it  is  clear that the equivariant moving frame map  described in Definition~\ref{def-eamf}  can be obtained  from a $SL_n(\K)\times\K$-equivariant matrix-completion   map  described in Definition~\ref{def-emc}. For convenience, we formulate the following trivial but useful theorem.
\begin{theorem}\label{thm-trivial} If $\mc\colon\rv\to \pmat$ is a matrix-completion map, then $\mfm\colon\rc\to\pmat$ defined by
$$\mfm(\cc)=\mc(\cc')$$
is a moving frame map. Moreover:
\begin{enumerate}
\item  If $\mc$ is of minimal degree, then so is $\mfm$ (see Definitions~\ref{def-mc} and~\ref{def-mf}).
\item If  $\mc$ is $SL_n(\K)\times\K$-equivariant, then   $\mfm$ is equi-affine (see Definitions~\ref{def-emc} and~\ref{def-eamf}).
\end{enumerate}
\end{theorem}

 In Section~\ref{subs-eqmc}, we present a constructive proof of the existence  of an equivariant matrix-completion   map. We base this proof on a general equivariantization method presented in the next section.  
\subsection{A process of equivariantization}\label{subs-equiv}
 The equivariantization process presented here  is a variation of Fels and Olver's  invariantization process \cite{FO99}.  For an arbitrary group action it is summarized in Proposition~\ref{prop-eqv-map}, while Proposition~\ref{prop-G1G2} addresses the situation  when the acting group is a direct product. 
\begin{definition}\label{def-sect} 
Let  a group  $\gva$  act on a set $\zva$. 
A map $\rho\colon \zva\to \gva$ will be called a \emph{$\gva$-equivariant section}  if it is equivariant under the  $\gva$-action on $\zva$  and    the $\gva$-action \eqref{self-act} on itself, i.e.
\beq\label{geq-rho} \underset{\pz\in\zva}{\forall}\ \underset{\g\in \gva }{\forall}\rho\left( \g\cdot
\pz\right)  =\g\star\,\rho\left(  \pz\right). \eeq
The map $\pi \colon \zva \to \zva$ defined by 
\beq\label{pi}\pi(\pz)=\rho (\pz)^{-1}\cdot\pz\eeq
will be called the \emph{$\rho$-canonical map} and its image $\cs$ will be called the \emph{$\rho$-canonical  set}.

\end{definition}
 Both $\cs$ and $\pi$ are determined by $\rho$, and a different choice of $\rho$ leads to a different  canonical map and set. 

\begin{remark} \label{rem-FO} \rm A reader familiar with  Fels and Olver paper \cite{FO99} will notice that our definition appears  to recall several key notions from that paper under {\em different names}. One of the reasons to  introduce a different terminology is to underscore
the following \emph{subtle but important difference}. In \cite{FO99}, the construction is done in the category of smooth manifolds: a Lie group $\gva$  acts  on a smooth manifold $\zva$  and the map $\rho$ is  required to be smooth. In the smooth setting,   $\rho$ and, therefore, $\pi$,  might exist only locally. Our construction  is done in the category of sets: no additional structure is assumed on the group~$\gva$ or the sets $\zva$ 
and~$\cs$. Maps $\rho$ and $\pi$  are defined everywhere on $\zva$ and are not required to respect any structure on the underlying sets. The following ``dictionary'' provides a further justification  for our terminology choice and comparison with \cite{FO99}.
\begin{itemize}
\item The map $\rho$  is called a \emph{left moving frame} in \cite{FO99} because a classical  geometric moving frame, such as of Frenet or Darboux, can be viewed as an equivariant map from an appropriate jet space  to  an appropriate  group. We choose a different name to avoid a confusion with the notion of an equi-affine frame defined in this paper. Our choice of the name is  informed by the following consideration:  the map $(\rho, Id_\zva)\colon\zva\to\gva\times \zva$ is a section of a trivial bundle $\gva\times \zva\to\zva$, equivariant under the action of $\gva$ on  $\zva$ and  the action of $\gva$ on~$\gva\times \zva$, where the latter action is defined by $\g\cdot(\tilde \g,\pz)=(\g\star\tilde \g, \g\cdot \pz)$.
\item The set $\cs$ is called a \emph{cross-section} in  \cite{FO99} because it is required to intersect  orbits transversally. In our setting, we show, in  Lemma~\ref{lem-pi}, that  for all $\pz\in\zva$, the orbit $\orb_\pz$ intersects~$\cs$ at the unique  point $\pi(\pz)$. Thus $\pi(z)$ can be viewed  as a canonical representative of the orbit $\orb_\pz$. This justifies our choice of calling $\cs$ a  canonical set and  calling $\pi$ a canonical map.\end{itemize}
 From the methodology point of view, Fels-Olver's construction usually starts with a choice of a cross-section  $\cs$, while   condition \eq{pi} \footnote{In fact,  Fels and Olver use condition  $\rho (\pz)\cdot\pz\in\cs$  to define, what they call, a \emph{right  moving frame},  while \eq{pi} defines a \emph{left moving frame}. Left and right frames are  related  by the group-inversion map. Classical differential geometric frames are left frames.} is used to  implicitly define the corresponding equivariant map~$\rho$.  Our approach is to directly  construct an equivariant   $G=\slnk\times\K$-section $\rho\colon \rv\to G$ (which turns out to be a discontinuous map)  using a variation of inductive construction \cite{kogan01,kogan03}.
\qq
\end{remark}

Keeping in mind the differences outlined in Remark~\ref{rem-FO} above, we present a relevant adaptation of the results  presented in \cite{FO99} in the following two lemmas:
\begin{lemma}[freeness]\label{lem-free} For an action of $\gva$ on $\zva$, a $\gva$-equivariant section $\rho\colon\zva\to\gva$ exists if and only if the action is free.
\end{lemma} 
\begin{proof}
\begin{itemize}
\item[($\Rightarrow$)] Assume an equivariant section $\rho\colon\zva\to\gva$ exists. Let $\pz\in\zva$ and $\g\in \gva$ be such that~$\g\cdot\pz=\pz$. Then $ \rho(\pz)=\rho(\g\cdot\pz)= \g\star \rho(\pz)$ and so $\g=e$, implying the freeness of the action.

\item[($\Leftarrow$)] Assume the action is free.  Let us choose one element in each orbit (which we can call a canonical point)  and denote the set of such canonical points as $\cs$. Let $\orb_\pz$ be the orbit of~$\pz\in \zva$ under the $\gva$-action.
Since the orbits partition $\zva$, for any $\pz\in\zva$, the set  $\cs\cap\orb_\pz$ consists of a single point and so  the map $\pi\colon\zva\to \cs$   sending a point $\pz\in\zva$ to $\cs\cap\orb_\pz$ is well defined and is $\gva$-invariant. For each $\pz\in \zva$,  since $\pz$ and $\pi(\pz)$ belong to the same orbit there is an element in $\gva$, sending $\pi(\pz)$ to $\pz$, and due to freeness of the action, such element is unique. Thus  a  map $\rho\colon\zva\to \gva$ is well defined by the condition $\pz=\rho(\pz)\cdot\pi(\pz)$.  To show its equivariance observe that for all $\pz\in \zva$ and $g\in\gva$:
$$ \pi(z)=\pi(g\cdot\pz)\,\, \Longrightarrow\,\, \rho(\pz)^{-1}\cdot \pz=\rho(g\cdot\pz)^{-1}\cdot (g\cdot\pz)\,\, \Longrightarrow\,\, \pz=(\rho(\pz)\star \rho(g\cdot\pz)^{-1}\star g)\cdot\pz.$$
The freeness of the action implies that $\rho(\pz)\star \rho(g\cdot\pz)^{-1}\star g=e$, which in turn implies the equivariance of $\rho$:
$$ \rho(g\cdot\pz)=\g\star\rho(\pz).$$
\end{itemize}

\end{proof}

\begin{lemma}\label{lem-pi} 
Let a group  $\gva$  act on a set $\zva$. 
Let $\rho\colon \zva\to \gva$ be a $\gva$-equivariant section.
Let  $\pi$  be the $\rho$-canonical map and  let  $\cs$ be the
 $\rho$-canonical set (see Definition~\ref{def-sect}). 
Then:
\begin{enumerate} 
\item  $\pi$ is $\gva$-invariant.
\item  $\pi$ is a projection on  $\cs$.
 
\item $\pi(\pz)=\cs\cap \orb_\pz$, where $ \orb_\pz$ is the orbit of $\pz$ under the $\gva$-action.
\item $\rho(\pz)=e$ if and only if $\pz\in \cs$.
\end{enumerate} 
\end{lemma}
\begin{proof}
\begin{enumerate}
\item To prove $\gva$-invariance of $\pi$, observe that for all  $\pz\in \zva$ and  all $\g\in \gva$
\begin{align*} \pi( \g\cdot\pz)&=\rho(\g\cdot \pz)^{-1} \cdot(\g\cdot \pz)=(\g\star \rho(\pz))^{-1} \cdot(\g\cdot \pz) \\
&=\left(\rho(\pz)^{-1} \star g^{-1} \right)\cdot(\g\cdot \pz)=\rho (\pz)^{-1}\cdot\pz=\pi (\pz),
\end{align*}
where we used the $\gva$-equivariance \eq{geq-rho} of  $\rho$ and the associativity property  of an action (see Definition~\ref{def-act}).  
\item To prove that $\pi$ is a projection on its image, observe that
$$ \pi(\pi(\pz))=\pi\left(\rho (\pz)^{-1}\cdot\pz\right)=\pi(\pz)$$
due to the $\gva$-invariance of $\pi$.
\item From \eq{pi} and the fact that $\rho(\pz)\in \gva$, it follows that $\pi(\pz)\in \orb_\pz$. Since $\cs$ is defined as the image of $\pi$, we also have  $\pi(\pz)\in \cs$. Thus $\pi(\pz)\in \cs\cap \orb_\pz$. On the other hand, assume that $\bpz\in
 \cs\cap \orb_\pz$. Then   $\bpz\in \orb_\pz$ and so there exists $\g\in \zva$ such that $\bpz=g\cdot\pz$. Since $\pi$ is projection on $\cs$, we have
 $$\bpz=\pi(\bpz)=\pi(g\cdot\pz)=\pi(\pz),$$
 where  the last equality is due to the $\gva$-invariance of $\pi$. Thus $\pi(\pz)$ is the \emph{unique} point of the intersection of  $\orb_\pz$ with  $\cs$. 
 \item   Assume $\rho(\pz)=e$. Then $\pi(\pz)=\rho(\pz)^{-1}\cdot\pz=\pz$ and so $\pz\in \cs$. Conversely, assume $\pz\in \cs$. Since $\pi$ is a projection  $\pi(\pz)=\rho(\pz)^{-1}\cdot \pz=\pz$.  From the freeness of the action (Lemma~\ref{lem-free}) $\rho(\pz)^{-1}=e$ and so~$\rho(\pz)=e$.
    
\end{enumerate}
\end{proof}
\begin{remark}[geometric interpretation]\rm
Figure~\ref{fig-rho} is a schematic illustration of   Definition~\ref{def-sect} and Lemma~\ref{lem-pi}.  The orbits of a $\gva$-action on $\zva$ partition the set $\zva$ into equivalence classes.  Let~$\rho$ be a $\gva$-equivariant section, $\pi$  be the corresponding canonical map and $\cs=Im\pi$ be the corresponding canonical set.  Each orbit intersects $\cs$ at the unique point that can be viewed as a canonical representative of the orbit. The canonical map $\pi$ sends the entire orbit to  this  canonical representative:  $\pi(\pz)=\pi(\g\cdot\pz)=\orb_\pz\cap\cs$ for all $\pz\in\zva$ and $\g\in \gva$.  It follows from \eq{pi}  that   $\pz=\rho(\pz)\cdot\pi(z)$ for all $\pz\in\zva$.

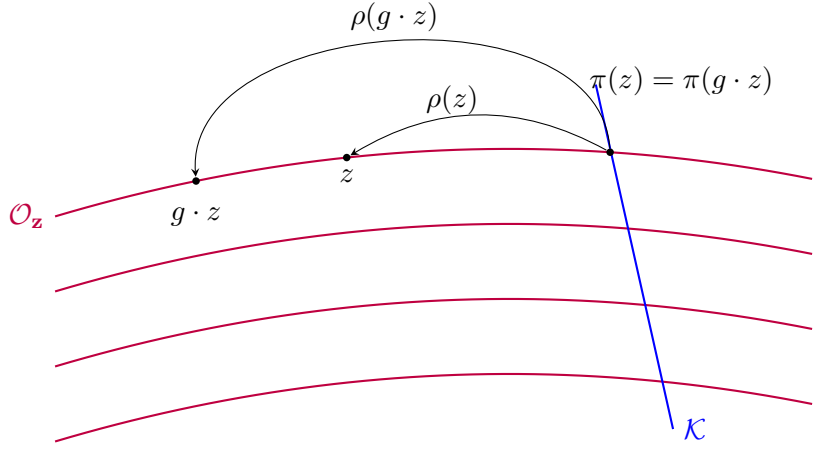
\begin{figure}[h!]

\centering
\begin{tikzpicture}
\begin{axis}[
  axis lines = none,
  xmin = -3, xmax = 0, ymin = -1, ymax = 4,
  domain = -3:2,
  restrict y to domain = -1:4,
  width = 3.0in,
  samples = 50,
  clip=false
  ]
  \foreach\z in {0,...,2} \addplot[purple, thick]{-0.1*x^2+\z};
  
  \addplot[purple, thick, name path = curve]{-0.1*x^2+3}
  coordinate[pos = .2]   (p1)
  coordinate[pos = .4]   (p2)
  coordinate[pos = .735] (pt)
  node[left,pos=0]{${\mathcal O}_{\bf z}$};
  
  \addplot[blue, thick, name path = blueline]{-9*(x-1)}
  node[right]{$\mathcal K$};
  
  
  \begin{scope}[every node/.style={circle, fill, inner sep=1pt}]
    \node (p1n) at (p1) [label={below:$\g\cdot\pz$}] {};
    \node (p2n) at (p2) [label={below:$\pz$}] {};
    \node (ptn) at (pt) [label={above right:$\pi(\pz)=\pi(\g\cdot\pz)$}] {};
  \end{scope}

    \draw [<-, >=stealth, bend left=90] (p1n) to node[above=] {$\rho(\g\cdot\pz)$} (ptn);
    \draw [<-, >=stealth, bend left] (p2n) to node[pos=.4, above=-2.8pt] {$\rho(\pz)$} (ptn);
\end{axis}
\end{tikzpicture}
\caption{Equivariant section $\rho$, canonical map $\pi$, and canonical set $\cs$.}
\label{fig-rho}
\end{figure}    
\noindent \qq \end{remark}
The following proposition  shows that if  an action of  $\gva$ on $\zva$ admits a $\gva$-equivariant section then any map from $\zva$ to another set, where an action of $\gva$ is defined, can be turned into a $\gva$-equivariant map\footnote{{A similar result, developed independently of our work,  is used in a recent preprint  \cite{olver2024} to obtain equivariant maps relevant for the development of equivariant neural networks \cite{lim-nelson23}.}}. 
\begin{proposition}[equivariantization]\label{prop-eqv-map} Assume  we have actions of a group $\gva$ on two sets $\zva$   and $\yva$ and a map  $f\colon\zva\to\yva$.  Assume there exists an  equivariant  section $\rho\colon\zva\to\gva$. Let~$\pi\colon\zva\to\zva$ be the corresponding canonical map and $\cs$ be its image. 

Then the map
$$\tilde f(\pz)=\rho(z)\cdot f(\pi(\pz))$$
is the unique  $\gva$-equivariant map such that its restriction to $\cs$ agrees with $f$, i.e. $\tilde f|_\cs=f|_\cs$.%
\end{proposition}
\begin{proof}
We first show that $\tilde f$ is equivariant  (see Definition~\ref{def-eqv}):
\begin{align*}\tilde f(\g\cdot \pz)&=\rho(\g\cdot\pz)\cdot f(\pi(\g\cdot\pz))=\left(\g\star\rho(\pz)\right)\cdot f(\pi(\pz))=\g\cdot\left(\rho(\pz)\cdot f(\pi(\pz))\right)
=\g\cdot \tilde f(\pz),
\end{align*}
where the second equality follows from  the equivariance of $\rho$ and the invariance of  $\pi$, while the third equality follows from the associativity property of a group action.

Since  for any $\bpz\in\cs$, we have $\pi(\bpz)=\bpz$ and $\rho(\bpz)=e\in \gva$ by Lemma~\ref{lem-pi}, we immediately have that $\tilde f|_\cs=f|_\cs$. 

To show  uniqueness, let  $\hat f\colon\zva\to \yva$ be another $\gva$-equivariant map such that $\hat f|_\cs=f|_\cs$. We note that  $\pz=\rho(\pz)\cdot \pi(\pz)$ by \eq{pi}.
Then
$$\hat f(\pz)=\hat f\left(\rho(\pz)\cdot \pi(\pz)\right)=\rho(\pz)\cdot \hat f(\pi(\pz))=\rho(\pz)\cdot f(\pi(\pz))=\tilde f(\pz),$$
where the second equality is due to the equivariance of $\hat f$  and the third is due to the facts that~$\pi(\pz)\in \cs$ and $\hat f|_\cs=f|_\cs$.
\end{proof}

We note the following useful corollary:
\begin{corollary}\label{cor-eqv} Under the assumptions of Proposition~\ref{prop-eqv-map}:
\begin{enumerate}
\item The map $f$ is $G$-equivariant if and only if  $f(\pz)=\rho(z)\cdot f(\pi(\pz))$ for all $\pz\in \zva$.
\item The map $f$ is $G$-invariant if and only if  $f(\pz)= f(\pi(\pz))$ for all $\pz\in \zva$. 
\end{enumerate}

\end{corollary}

In general, the explicit construction of an equivariant section can be very challenging. The following results allow us to significantly  facilitate this construction  when  $\gva$ is a direct product (see Remark~\ref{rem-act}). This is a variation of the inductive construction presented in \cite{kogan01,kogan03} (see also subsequent  generalizations in \cite{O-rec11,V-ind13, OV-rec18}).  
\begin{lemma}\label{lem-direct1}    Let two groups $\gva_1$ and $\gva_2$ act on a set $\zva$, such that  each action admits an equivariant section, $\rho_1\colon\zva\to\gva_1$ and $\rho_2\colon\zva\to\gva_2$, respectively,  with the corresponding canonical maps $\pi_1\colon \zva\to \zva$  and  $\pi_2\colon \zva\to \zva$.  Assume the actions of $\gva_1$ and $\gva_2$ commute, therefore, defining an action of  the direct product $\gva=\gva_1\times\gva_2$ on $\zva$. 

 \emph{Then the following  statements are equivalent:}
\begin{enumerate}
\item $\rho=(\rho_1,\rho_2)$ is a $\gva$-equivariant section. 
\item $\rho_1$ is a $\gva_2$-invariant map and $\rho_2$ is  a $\gva_1$-invariant map.

\end{enumerate}
\end{lemma}
\begin{proof}
The map $\rho=(\rho_1,\rho_2)$ is a $\gva$-equivariant section if and only if  for  any $g_1\in \gva_1$ and $g_2\in \gva_2$:
\beq\label{eq-12}\rho\big((g_1,g_2)\cdot \pz\big)=(g_1,g_2)\cdot \rho(\pz).\eeq
Expanding the left-hand side of \eq{eq-12} using commutativity of the $\gva_1$ and $\gva_2$ actions, as well as  equivariance properties of sections $\rho_1$ and $\rho_2$, we have:
 $$\big(\rho_1 (g_1\cdot  g_2\cdot \pz), \rho_2 (g_1\cdot  g_2\cdot \pz)\big)=\big(g_1\star \rho_1 ( g_2\cdot \pz),g_2\star \rho_2 (g_1\cdot \pz)\big),$$
 while the right-hand side of \eq{eq-12} expands to $\big(g_1\star \rho_1 ( \pz),g_2\star \rho_2 ( \pz)\big)$. Equality  of the sides is equivalent to 
  $$   \rho_1 ( g_2\cdot \pz)= \rho_1 ( \pz) \text{ and }   \rho_2 ( g_1\cdot \pz)= \rho_2 ( \pz),$$
  which is, by definition, is  the second statement of the Lemma.
 \end{proof}

\begin{lemma}\label{lem-direct2}     Under the assumptions of Lemma~\ref{lem-direct1}:

\begin{enumerate}
\item If $\rho_1$ is a $\gva_2$-invariant map and $\rho_2$ is  a $\gva_1$-invariant map, then  $\pi_1\circ\pi_2=\pi_2\circ\pi_1$.
\item If  the action of $\gva$ is free\footnote{Warning: freeness and commutativity of the $\gva_1$ and $\gva_2$-actions does not imply freeness of the $\gva_1\times \gva_2$-action. (For a trivial counter-example,  consider the case when $\gva_1=\gva_2$ is a commutative group acting freely.)}, the converse is also true.
 \end{enumerate}
\end{lemma}
\begin{proof}
\begin{enumerate}
\item  From the second part of Corollary~\ref{cor-eqv}  we have
 \beq\label{eq-23} \rho_1(\pz)=\rho_1(\pi_2(\pz))\text{ and }  \rho_2(\pz)=\rho_2(\pi_1(\pz)).\eeq
 Therefore,
$$ \pi_1\circ\pi_2 (\pz)= \rho_1(\pi_2(\pz))^{-1}\cdot\pi_2(\pz)= \rho_1(\pz)^{-1}\cdot\rho_2(\pz)^{-1}\cdot \pz$$ 
and
$$ \pi_2\circ\pi_1 (\pz)= \rho_2(\pi_1(\pz))^{-1}\cdot\pi_1(\pz)= \rho_2(\pz)^{-1}\cdot\rho_1(\pz)^{-1}\cdot \pz$$ 
Due to commutativity of the of $\gva_1$ and $\gva_2$ actions $\pi_1\circ\pi_2=\pi_2\circ\pi_1$.
 \item  Observe that
  $$ \pi_1\circ\pi_2 (\pz)= \rho_1(\pi_2(\pz))^{-1}\cdot\pi_2(\pz)= \rho_1(\pi_2(\pz))^{-1}\cdot\rho_2(\pz)^{-1}\cdot \pz=\big(\rho_1(\pi_2(\pz),  \,\rho_2(\pz)^{-1}\big)\cdot\pz$$
 and  similarly:
   $$ \pi_2\circ\pi_1 (\pz)=  \big(\rho_1(\pz)^{-1}, \rho_2(\pi_1( \pz))^{-1}\big)\cdot\pz.$$
   Therefore, $\pi_1\circ\pi_2=\pi_2\circ\pi_1$ is equivalent to:
   $$\big(\rho_1(\pi_2(\pz),  \,\rho_2(\pz)^{-1}\big)\cdot\pz=\big(\rho_1(\pz)^{-1}, \rho_2(\pi_1( \pz))^{-1}\big)\cdot\pz.$$

  The freeness assumption  allows us to ``cancel'' $\pz$ and conclude that
   \beq\label{eq-23} \rho_1(\pz)=\rho_1(\pi_2(\pz))\text{ and }  \rho_2(\pz)=\rho_2(\pi_1(\pz)).\eeq
    Then, by the second part of Corollary~\ref{cor-eqv}, 
    $\rho_1$ is $\gva_2$-invariant and    $\rho_2$ is $\gva_1$-invariant.
 \end{enumerate}
\end{proof}

Lemma~\ref{lem-direct1} shows that   a $\gva_1$-equivariant section, invariant under the $\gva_2$-action, and a $\gva_2$-equivariant section, invariant under the $\gva_1$-action, can be easily combined to construct a  $\gva_1\times\gva_2$-equivariant section.  
For the problem at hand,  in Theorem~\ref{thm-rho12}  below, we manage to construct a $\gva_2$-equivariant section  invariant under  the $\gva_1$-action, but a $\gva_1$-equivariant section, we construct,  lacks the desired $\gva_2$-invariance property. This makes constructing a $\gva_1\times\gva_2$-equivariant section  more challenging, but still feasible due to the following lemma and  proposition.
\begin{lemma}\label{lem-twist}
Let two groups $\gva_1$ and $\gva_2$ act on a set $\zva$.  Assume the actions of $\gva_1$ and $\gva_2$ commute. Let  $\rho_2\colon\zva\to\gva_2$ be an equivariant section and  $\pi_2\colon \zva\to \zva$ be  the corresponding canonical map.  Assume  further that the map $\rho_2$ is $\gva_1$-invariant. Then:
\begin{enumerate}
\item The canonical map $\pi_2$ is $\gva_1$-equivariant.
\item  \label{cs2-inv} The image $\cs_2$ of $\pi_2$ is a $\gva_1$-invariant subset of $\zva$ is the sense that $\g_1\cdot \bpz\in\cs_2$ for any $\bpz\in\cs_2$ and $\g_1\in \gva_1$. 
\item \label{p-twist} If $\rho_1\colon \zva\to \gva_1$ is a $\gva_1$-equivariant section, then $\tilde\rho_1=\rho_1\circ\pi_2$ is also a $\gva_1$-equivariant section,  invariant under $\gva_2$-action.
\end{enumerate}
\begin{proof}
\begin{enumerate}
\item  To show $\gva_1$-equivariance of $\pi_2$,  observe that for  all $\g_1\in \gva_1$  and $\pz\in \zva$:
\begin{align*} \pi_2(\g_1\cdot\pz)&=\rho_2(\g_1\cdot\pz)^{-1}\cdot (\g_1\cdot\pz)=\rho_2(\pz)^{-1}\cdot (\g_1\cdot\pz)=\g_1\cdot (\rho_2(\pz)^{-1}\cdot\pz)=\g_1\cdot \pi_2(\pz),
\end{align*}
where, in the second equality, we used $\gva_1$-invariance of $\rho_2$ and, in the third equality, we used the commutativity of  $\gva_1$- and $\gva_2$-actions.
\item Since $\pi_2$ is a projection, then $\bpz=\pi_2(\bpz)$ for any $\bpz\in \cs_2$. Hence, for any $\g_1\in \gva_1$:
$$ \g_1\cdot\bpz= \g_1\cdot\pi_2(\bpz)= \pi_2(\g_1\cdot\bpz)\in\cs_2,$$
where the last quality is due to $\gva_1$-equivariance of $\pi_2$ proven above.
\item To prove that  $\tilde\rho_1$  is a $\gva_1$-equivariant section, observe that for all $\g_1\in\gva_1$ and $\pz\in \zva$ due to  $\gva_1$-equivariance of both $\rho_1$ and $\pi_2$ we have:
$$\tilde\rho_1(\g_1\cdot\pz)=\rho_1(\pi_2(\g_1\cdot \pz))=\rho_1(\g_1\cdot \pi_2( \pz))=\g_1\cdot\rho_1( \pi_2( \pz))=\g_1\cdot\tilde\rho_1(\pz).$$
Invariance of $\tilde\rho_1$ under the $\gva_2$-action follows from $\gva_2$-invariance of $\pi_2$.
\end{enumerate}
\end{proof}

\end{lemma}

\begin{proposition}[direct product]\label{prop-G1G2} 
Under the assumptions and with the notations of Lemma~\ref{lem-direct1}, let  $\cs_1$ and $\cs_2$  be the images of $\pi_1$ and $\pi_2$, respectively. Assume  further that the map $\rho_2$ is $\gva_1$-invariant.
%
%
 Then:
\begin{enumerate}
\item\label{ind-rho} The map $\rho\colon\zva\to \gva $ defined by 
$$ \rho(\pz)=\big(\rho_1(\pi_2(\pz)), \rho_2(\pz)\big)$$ a $\gva$-equivariant section. 
\item  \label{ind-pi} If $\pi$ is  $\rho$-canonical map and $\cs$ is the image of $\pi$, then\footnote{Warning:  maps $\pi_1$ and $\pi_2$, in general, do not commute, as it can been seen in Theorem~\ref{thm-rho12}  below. See also Lemma~\ref{lem-direct2}.}
$$\pi=\pi_1\circ\pi_2 \quad{and}\qquad \cs=\cs_1\cap\cs_2.$$ 
\end{enumerate}
\end{proposition}

\newpage
\begin{proof}
\begin{enumerate}

\item  By Part~\ref{p-twist} of Lemma~\ref{lem-twist}, $\tilde\rho_1=\rho_1\circ\pi_2$ is  a $\gva_1$-equivariant section, which is, at the same time, is   a $\gva_2$-invariant map.
Then $\rho(\pz)=\big(\tilde\rho_1,\rho_2\big)$ is a $\gva$-section by  Lemma~\ref{lem-direct1}.
\item  We unwind the definitions:
\begin{align*} &\pi(\pz)=\rho(\pz)^{-1}\cdot \pz=\big(\rho_1(\pi_2(\pz)), \rho_2(\pz)\big)^{-1}\cdot\pz=\rho_1(\pi_2(\pz))^{-1} \cdot\big(\rho_2(\pz)^{-1}\cdot \pz\big)\\
&=\rho_1(\pi_2(\pz))^{-1} \cdot \pi_2(\pz)=\pi_1(\pi_2(\pz)).
\end{align*}
From the above formula, it  immediately  follows that $\cs=Im(\pi)\subset Im(\pi_1)=\cs_1.$ We also note that since $\pi_2(\pz)\in\cs_2$ and $\cs_2$ is a $\gva_1$-invariant subset of $\zva$, as proven in Part~\ref{cs2-inv} of Lemma~\ref{lem-twist}, then 
$$\pi(\pz)=\pi_1(\pi_2(\pz))=\rho_1(\pi_2(\pz))^{-1} \cdot \pi_2(\pz)\in\cs_2.$$
Thus, we have shown that $\cs\subset\cs_1\cap\cs_2$.
Conversely,  for any  $\pz\in\cs_1\cap\cs_2$, since $\pi_1$ and~$\pi_2$ are projections on $\cs_1$ and $\cs_2$, respectively,  we have $\pz=\pi_1(\pz)=\pi_2(\pz)$. Therefore, $\pi(\pz)=\pi_1(\pi_2(\pz))=\pz$, and so $\pz\in Im(\pi)=\cs$, proving the reverse inclusion  $\cs\supset\cs_1\cap\cs_2$. We conclude that  $\cs=\cs_1\cap\cs_2$.
\end{enumerate}
\end{proof}

\subsection{\rm{\bf {EMC}} from a matrix-completion  (\rm{\bf{MC}})}\label{subs-eqmc}
In this section, we apply the general theory developed in Section~\ref{subs-equiv} to turn an arbitrary matrix-completion map into an $\slnk\times \K$-equivariant matrix-completion map, where the $\slnk\times \K$-action is defined by \eq{Ls-act}. 

 \begin{theorem}[equivariant matrix-completion]\label{thm-eq-mc}  Let 
\begin{itemize}
\item $\mc\colon \rv\to  \K[t]^{n\times n}$  be a matrix-completion map;
\item   $\rho\colon\rv \to \gva=\slnk\times \K$ be an equivariant section for the   $\slnk\times\K$-action on $\rv$ as defined by  \eq{Ls-act};
\item  $\pi\colon\rv\to\rv$ be the corresponding canonical map with the image $\cs$.
\end{itemize}
 Then the map $\emc\colon \rv\rightarrow \K[t]^{n\times n}$ 
defined by 
\beq\label{eq-mc}\emc(\vv)=\rho(\vv)\cdot \mc\left(\rho(\vv)^{-1}\cdot \vv\right)\eeq
 is  the unique $\slnk\times \K$-equivariant matrix-completion map such that $\emc(\bar \vv)=\mc(\bar \vv)$ for all $\bar \vv\in\cs$.
\end{theorem}
\begin{proof} The fact that the map  $\emc\colon \rv\rightarrow \K[t]^{n\times n}$ is  the unique $\slnk\times \K$-equivariant \emph{map} such that $\emc(\bar \vv)=\mc(\bar \vv)$ for all $\bar \vv\in\cs$ follows from Proposition~\ref{prop-eqv-map}. To show that it is a \emph{matrix completion} map we can use the same argument as in Lemma~\ref{lem-mc} with $L$ replaced by $L(\vv)$ and $s$ replaced with $s(\vv)$, where $\rho(\vv)=\left(L(\vv),s(\vv)\right)$.
\end{proof}

Now the problem boils down to an explicit  construction of  an equivariant section $$\rho\colon\rv\to \slnk\times\K.$$  
To  take advantage of Proposition~\ref{prop-G1G2} with $\gva_1=\slnk$ and $\gva_2=\K$ actions on $\zva=\rv$ given by \eq{L-act} and \eq{s-act}, respectively, we need to construct maps $\rho_1\colon\rv\to \slnk$  and $\rho_2\colon\rv\to\K$, such that:
\begin{align}
\label{eqv-rho1} \underset{\vv\in\rv}{\forall}\ \underset{\l\in SL_{n}\left(\K\right) }{\forall}\rho_1\left( \l\cdot\vv\right)  &=\l\,\rho_1\left(  \vv\right),\\
\label{eqv-rho2} \underset{\vv\in\rv}{\forall}\ \underset{s\in \K }{\forall}\rho_2\left( s\cdot\vv\right)  &=s+\rho_2\left(  \vv\right),\\
\label{inv-rho2}\underset{\vv\in\rv}{\forall} \underset{\l\in \slnk}{\forall}\rho_2\left(\l\cdot\vv\right)&=\rho_2(\vv).
\end{align} 
Moreover,  to insure that we stay within the field $\K$, we will seek to construct $\rho_1$ and $\rho_2$  which are  (piecewise) rational function in  the coefficients of  the polynomial vectors in $\rv$. We achieve this in the following proposition. An interesting twist is that we use  map $\rho_1$ to construct map~$\rho_2$.

\begin{theorem}[equivariant sections]\label{thm-rho12}
Let $\vv$ be a regular polynomial vector of degree $d$: 
\beq\label{v}
\vv=\underset{\cv}{\underbrace{\left[
\begin{array}
[c]{ccc}%
v_{10} & \dots & v_{1d}\\
\vdots & \vdots & \vdots\\
v_{n0} & \dots & v_{nd}
\end{array}
\right]  }}\left[
\begin{array}
[c]{c}%
t^{0}\\
\vdots\\
t^{d}%
\end{array}
\right],
\eeq
where $V$,  is the coefficient matrix of $\vv$.
Let %
\beq\label{bQ}
\overline{\cv}=\left[
\begin{array}
[c]{cccc}%
v_{1i_{1}} & \dots & v_{1i_{n-1}} & v_{1i_{n}}\\
\vdots & \vdots & \vdots & \vdots\\
v_{ni_{1}} & \dots & v_{ni_{n-1}} & v_{ni_{n}}
\end{array}
\right] 
\eeq
be the most right submatrix of $\cv$ of rank $n$. (We note that $i_n=d$.)
Let $k$ be the index of the right most column of $\cv$
not included in $\overline{\cv}$, or  equivalently
\beq \label{eq-kv}
k=\min_{j=1,\dots,n}\{i_{j}|i_{j}=d-n+j\}-1.
\eeq
\begin{enumerate} 
\item 
 Define $\rho_1(\vv)$ to be the matrix obtained from
$\overline{\cv}$ by dividing each element of its last column by $|\overline\cv|$.  Then the map   $\rho_1\colon
\rv\rightarrow SL_{n}(\K)$ is 
$\slnk$-equivariant section as defined by \eq{eqv-rho1}.
\item Define  $\rho_2(\vv)$ to be the number obtained by  
\beq\label{eq-rho2}
\rho_2(\vv)=\frac{\;\;\;\;\;\;{\rm coeff}(\bv_{n-(d-k-1)},t,k)}{(k+1)\;{\rm coeff}(\bv_{n-(d-k-1)},t,k+1)},
\eeq
where the polynomial vector  $\bv=\rho_1(\vv)^{-1}\cdot  \vv$.
Then the map $\rho_2\colon \rv\rightarrow \K$ is both $\K$-equivariant as defined by \eq{eqv-rho2} and $\slnk$-invariant as defined by \eq{inv-rho2}.
\item Let $\pi_1,\pi_2\colon\rv\to\rv$ be the two maps defined by  $\pi_i(\vv)=\rho_i(\vv)^{-1}\cdot\vv$. Then  the map $\rho\colon \rv\to \slnk\times\K$ defined by
              \beq\label{eq-product-rho}\rho(\vv)=\big(\rho_1(\pi_2(\vv)),\rho_2(\vv)\big)\eeq
              is an   $\slnk\times\K$-equivariant section  and  the coefficient matrix of the canonical form $\pi(\vv)=\rho(\vv)^{-1}\cdot\vv$ has the following shape:
\beq\label{canonical}\text{\rm
\renewcommand{\arraystretch}{1.3} 
\begin{tabular}{|cccccccccc|}
\cline{1-10}
$\ast$ &\multicolumn{1}{c|}{1} & & & & & & & &   \\
\cline{3-4}
$\ast$ &0 & $\ast$ &\multicolumn{1}{c|}{1}  & & & & & &    \\
$\vdots$&$\vdots$&$\vdots$ &$\vdots$ &  $\ddots$ & & & & &   \\
\cline{6-7}
$\ast$& 0 & $\ast$ & 0 & $\cdots$  & $0$ & \multicolumn{1}{c|}{1}  &  & & \\
\cline{8-8}
$\ast$ & 0  & $\ast$ & 0 & $\cdots$ &  $*$ & 0 &  \multicolumn{1}{c|}{1} & & \\ 
$\vdots$ & $\vdots$ & $\vdots$ & $\vdots$&   &  $\vdots$ & $\vdots$ & $\vdots$ &  $\ddots$  & \\ 
\cline{10-10}
$\ast$ & 0 & $\ast$ & 0 & $\cdots$ &  $*$ & 0& 0  & $\cdots$ &  $|\bar{V}|$  \\
\cline{1-10}
\multicolumn{1}{}{}& $i_1$ & & $i_2$  &  &   $k$  & $k+1$ & $k+2$ &  &   \multicolumn{1}{r}{$d$}
\end{tabular}}
\eeq

Columns are numbered from~$0$ to $d$ (left-to-right). Rows are numbered from~$1$ to $n$ (top-to-down). 
Steps are numbered from~$1$ to~$n$ (left-to-right).
The depth of $j$-th step is $i_j-i_{j-1}$ where $i_0=-1$.
For $j=1,\dots,n-1$, the $1$ in the $j$-th row is located in the $i_j$-th column.
The $k$-th column is located in the $(n-(d-k-1))$-th step. Arbitrary entries are denoted by~*. Their expressions in terms of   the coefficients of $\vv$ are invariant under the~$\slnk\times\K$-action.
\end{enumerate}
\end{theorem}
\begin{proof} Throughout the proof, since the matrix $\cv$, its submatrix $\overline \cv$, and the integer $k$ depend on the input vector $\vv\in \rv$,  we will  treat them as maps $\cv\colon \rv\to \K^{n\times (d+1)}$, $\overline{\cv}\colon \rv\to \K^{n\times n}$, and  $k\colon  \rv\to \{0,\dots, d-1\}$.
\begin{enumerate}
\item We observe that  $\rho_1$ is defined because the components of the polynomial vector $\vv$ are linearly independent.  To prove that $\rho_1$ is $\slnk$-equivariant as defined by \eq{eqv-rho1}, we note that for any 
 $\l\in \slnk $
\[
\l\cdot \vv = \underset{\l \cv}{\underbrace{\left[
\begin{array}
[c]{ccc}%
\l V_{*0} & \dots & \l V_{*d}\\
&  &
\end{array}
\right]  }}\left[
\begin{array}
[c]{c}%
t^{0}\\
\vdots\\
t^{d}%
\end{array}
\right]  ,
\]
where $V_{*i}$ denote the columns of matrix $\cv$ indexed by $i\in\{0,\dots,d\}$.
Then since $|\l|=1$:   
\begin{enumerate}
\item $\overline \cv(\l\cdot \vv)=\l \overline  \cv(\vv)$;

\item $k(\vv)=k(\l\cdot \vv)$;

\item $|\overline  \cv(\l\cdot \vv)|=|\overline \cv(\vv)|$.
\end{enumerate}
Therefore, $\rho_1(\l\cdot \vv)=\l\,\rho_1(\vv)$.
\item 
We observe that  $\rho_2$ is defined because $\deg\vv\geq n$, and so there exists  at least one column of $\cv$ not included in $\overline  \cv$.  We also note that, by construction,  $\rho_2(\vv)$  is completely determined by $\pi_1(\vv)= \rho_1(\vv)^{-1}\cdot\vv$.
Therefore,
\beq\label{rho2-pi1} \rho_2(\vv)=\rho_2(\pi_1(\vv)).\eeq
\begin{enumerate}
\item The  $\slnk$-invariance of $\rho_2$, described by \eq{eqv-rho2}, follows immediately from \eq{rho2-pi1} and the $\slnk$-invariance of $\pi_1$:
$$\rho_2\left(\l\cdot \vv\right) =\rho_2\left(\pi_1(\l\cdot \vv)\right)=\rho_2(\pi_1(\vv))= \rho_2( \vv).$$

\item Due to \eq{rho2-pi1}, to prove  $\K$-equivariance of $\rho_2$,  described by \eq{inv-rho2},  understanding the structure of $\pi_1(s\cdot \vv)$ is crucial,  and the following identity helps us to do so: 
\begin{align}\label{pi1sv} \pi_1(s\cdot \vv)=\pi_1\left(\rho_1(\vv)^{-1}\cdot \left(s\cdot \vv\right)\right)=\pi_1\left(s\cdot \left(\rho_1(\vv)^{-1}\cdot \vv\right)\right)= \pi_1(s\cdot\pi_1(\vv)),
\end{align}
where in the first equality we use   the  $\slnk$-invariance of $\pi_1$,   in the second, the commutativity of actions defined by \eq{s-act} and \eq{L-act}, while, in the last we use the definition of $\pi_1$.

By construction,
\beq \label{piv}\pi_1(\vv)=\begin{bmatrix} \bar\vv_1 \\ \vdots\\  \\ \vdots\\ \\ \vdots\\  \bar{\vv}_n\end{bmatrix}=\begin{bmatrix} 0 \\ \vdots\\0  \\t^{k+1}\\ t^{k+2}\\ \vdots\\  |\overline{\cv}| t^d\end{bmatrix}+\begin{bmatrix}0\\ \vdots\\  0 \\  b_{n-(d-k-1)} \\  b_{n-(d-k-2)} \\  \vdots\\ b_n\end{bmatrix}\,t^k+ \ww(t),\eeq
where $k$ is defined by \eqref{eq-kv}, $\overline{\cv}$ by \eq{bQ}, $b_{n-(d-k-1)},\dots, b_n\in \K$ are coefficients whose exact values in terms of the coefficients of $\vv$  are  irrelevant,  and  $\ww$ is a polynomial vector of degree less than $k$, whose explicit expression in terms of $\vv$ is also irrelevant.  Observe that the coefficient matrix of $\bar\vv=\pi_1(\vv)$ has a staircase shape as shown below.

\beq\label{vbar}\cv(\bar \vv)=\text{
\renewcommand{\arraystretch}{1.3} 
\begin{tabular}{|cccccccccc|}
\cline{1-10}
$\ast$ &\multicolumn{1}{c|}{1} & & & & & & & &   \\
\cline{3-4}
$\ast$ &0 & $\ast$ &\multicolumn{1}{c|}{1}  & & & & & &    \\
$\vdots$&$\vdots$&$\vdots$ &$\vdots$ &  $\ddots$ & & & & &   \\
\cline{6-7}
$\ast$& 0 & $\ast$ & 0 & $\cdots$  & $b_{n-(d-k-1)}$ & \multicolumn{1}{c|}{1}  &  & & \\
\cline{8-8}
$\ast$ & 0  & $\ast$ & 0 & $\cdots$ &  $b_{n-(d-k-2)}$ & 0 &  \multicolumn{1}{c|}{1} & & \\ 
$\vdots$ & $\vdots$ & $\vdots$ & $\vdots$&   &  $\vdots$ & $\vdots$ & $\vdots$ &  $\ddots$  & \\ 
\cline{10-10}
$\ast$ & 0 & $\ast$ & 0 & $\cdots$ &  $b_n$ & 0& 0  & $\cdots$ &  $|\bar{V}|$  \\
\cline{1-10}
\multicolumn{1}{}{}& $i_1$ & & $i_2$  &  &   $k$  & $k+1$ & $k+2$ &  &   \multicolumn{1}{r}{$d$}
\end{tabular}}
\eeq

Therefore,
\beq\label{rho2v}\rho_2(\vv)=\rho_2(\pi_1(\vv))=\begin{cases} \frac{b_{n-(d-k-1)}}{k+1} & \text{ if } k<d-1,\\ \frac{b_n}{|\overline{\cv}| d}& \text{ if }  k=d-1.\end{cases}\eeq
 As a side remark, we observe  that, in contrast with matrix \eq{canonical}, the entry~$b_{n-(d-k-1)}$  in the  $(n-(d-k-1))\times(n-(d-k-1))$-th position  of matrix \eq{vbar} does not have to be zero. When it is, then $\rho_2(\vv)=0$ and \eq{vbar} is the matrix of the canonical form~$\pi(\vv)$.

\medskip

The polynomial vector $s\cdot\pi_1(\vv)$, obtained by replacing $t$ with $t+s$ in \eq{piv}, is
%
\beq \label{spi1v}s\cdot\pi_1(\vv)=\begin{bmatrix} 0 \\ \vdots\\0  \\t^{k+1}\\ t^{k+2}\\ \vdots\\  |\overline{\cv}| t^d\end{bmatrix}+\begin{bmatrix}0\\ \vdots\\0 \\ b_{n-(d-k-1)}+(k+1)\,s  \\  p_1(t) \\  \vdots\\ p_{d-k-1}(t)\end{bmatrix}\,t^k+ \hat \ww(t),\eeq
where $p_i(t)$, $i=1,\dots, (d-k-1)$, are polynomials of degree $i$ and $\hat \ww$ is an updated polynomial vector of degree less than $k$, whose coefficients can be expressed in terms coefficients of  $\vv$ and $s$. 
 The coefficient matrix of $s\cdot\pi_1(\vv)$ is obtained by replacing zeros, located below  1s,  by some, in general non-zero, numbers  in the staircase matrix shown in \eq{vbar}. 
We observe that $\rho_1(s\cdot\pi_1(\vv))$ is a lower triangular matrix with 1s on the diagonal and so is its inverse. Therefore,
\beq \label{pi1spi1v}\pi_1(s\cdot\pi_1(\vv))=\rho_1(s\cdot\pi_1(\vv))^{-1}\cdot(s\cdot\pi_1(\vv))=\begin{bmatrix} 0 \\ \vdots\\0  \\t^{k+1}\\ t^{k+2}\\ \vdots\\  |\overline{\cv}| t^d\end{bmatrix}+\begin{bmatrix}0\\ \vdots\\0 \\  b_{n-(d-k-1)}+(k+1)\,s \\  \hat b_{n-(d-k-2)} \\  \vdots\\ \hat b_n\end{bmatrix}\,t^k+ \tilde \ww(t),\eeq
where $\hat b_{n-(d-k-2)},\dots, \hat b_n\in \K$ and  $\tilde \ww$ is a polynomial vector of degree less than $k$, whose explicit expression in terms of $\vv$ and $s$ are irrelevant. 

From \eq{rho2-pi1}, \eq{pi1sv}, and \eq{pi1spi1v}:
\begin{align*} \rho_2(s\cdot v) & = \rho_2(\pi_1(s\cdot v)) \\
                                & =\rho_2(\pi_1(s\cdot \pi_1(v)))\\
                                &=\begin{cases} \frac{b_{n-(d-k-1)} +(k+1)s}{k+1} & \text{ if } k<d-1\\ \frac{b_n+ds}{|\overline{\cv}| d}& \text{ if }  k=d-1\end{cases}\\
& =\begin{cases} \frac{b_{n-(d-k-1)} }{k+1} +s& \text{ if } k<d-1\\ \frac{b_n}{|\overline{\cv}| d}+s& \text{ if }  k=d-1\end{cases}=\rho_2(\vv)+s. 
\end{align*}

\end{enumerate}
\item Equivariance of the map $\rho$ follows from Proposition~\ref{prop-G1G2}, and from the same proposition we have
$$\pi(\vv)=\rho(\vv)^{-1}\cdot\vv=\pi_1(\pi_2(\vv))=\pi_1(\rho_2(\vv)^{-1}\cdot\vv)=\pi_1(\rho_2(\vv)^{-1}\cdot\pi_1(\vv)),$$
where the last equality follows from  \eq{pi1sv}  with $s=\rho_2(\vv)^{-1}$. Formula \eq{rho2v} is chosen so that  when we substitute $s=\rho_2(\vv)$ into  the entry $b_{n-(d-k-1)}+(k+1)\,s $ appearing in matrix \eq{spi1v} becomes 0.
As follows from the explanation under \eq{spi1v}, subsequent  application of $\pi_1$ will  result in the desired canonical form \eq{canonical}.  Non-constant entries of  matrix \eq{canonical} can be expressed as rational functions in the coefficients of $\vv$, and these functions are invariant with respect to  the $\slnk\times\K$-action, because the map $\pi$ is $\slnk\times\K$-invariant.

\end{enumerate}
\end{proof}
\begin{remark}[generic case]\rm
Generically,  the last $n$ columns of the coefficient matrix   $\cv$ in \eq{v} are linearly independent and so 
$[i_{1},\dots, i_{n}]=[d-(n-1),\dots,d]$, i. e. $i_{j}=d-n+j$ for $j=1,\dots{n}$. Then
\begin{itemize}

\item $\rho_1(\vv)=\left[V_{*d-n+1},\dots,\frac 1{|\overline{\cv}|} V_{*d}\right] $
 \item $k(\vv)=d-n.$
\end{itemize}

\noindent It follows that 
 $$\bar\vv=\pi_1(\vv)=\rho_1(\vv)^{-1}\cdot \vv=\begin{bmatrix} t^{d-n+1} \\  t^{d-n+2}\\ \vdots\\ |\overline{\cv}| \,t^d\end{bmatrix}+\begin{bmatrix} b_1 \\ b_2\\ \vdots\\  b_n\end{bmatrix}\,t^{d-n}+ \ww(t),$$ where $\ww$ is a polynomial vector of degree less than $d-n$ and so
 $$\rho_2(\vv)=\frac{b_1}{d-n+1}.$$
 The canonical form of $\vv$ is
 $$\pi(\vv)=\rho(\vv)^{-1}\cdot \vv=\begin{bmatrix} t^{d-n+1} \\  t^{d-n+2}\\ \vdots\\ |\overline{\cv}| \,t^d\end{bmatrix}+\begin{bmatrix} 0 \\ c_2\\ \vdots\\  c_n\end{bmatrix}\, t^{d-n}+ \hat\ww(t).$$
 \qq \end{remark}
\begin{example}[equivariant sections]\rm  To illustrate  Theorem~\ref{thm-rho12}, consider $$\vv= \left[\begin{array}{c}
                        t^{4}+t^{3}+2 t^{2}+1 
                        \\
                        2 t^{4}+t^{3}+3 t^{2}+2 
                        \\
                        3 t^{4}+t^{3}+4 t^{2}+5 t +3 
                \end{array}\right].$$
\noindent Note                
\begin{enumerate}
\item The  coefficient matrix of $\vv$ is $$\cv(\vv)=\left[\begin{array}{ccccc}
1 & 0 & 2 & 1 & 1 
\\
 2 & 0 & 3 & 1 & 2 
\\
 3 & 5 & 4 & 1 & 3 
\end{array}\right]
.$$
The rightmost full-rank submatrix $$\VV(\vv)=\left[\begin{array}{ccc}
                0 & 1 & 1 
                \\
                0 & 1 & 2 
                \\
                5 & 1 & 3 
        \end{array}\right]
$$ consists of  columns indexed by $[i_1,i_2,i_3]=[1,3,4]$ 
(where the smallest column-index of $V$ is $0$). 
Thus we have $k=2$. Since $|\VV|=5$, we have $$\rho_1(\vv)= \left[\begin{array}{ccc}
                0 & 1 & \frac{1}{5} 
                \\
                0 & 1 & \frac{2}{5} 
                \\
                5 & 1 & \frac{3}{5} 
        \end{array}\right]
        .$$
\item Observe that $$\bar \vv=\pi_1(\vv)=(\rho_1)^{-1}\vv=\begin{bmatrix}\bar\vv_1\\\bar\vv_2\\\bar\vv_3\end{bmatrix}=\left[\begin{array}{c}
                t  
                \\
                t^{3}+t^{2} 
                \\
                5 t^{4}+5 t^{2}+5 
        \end{array}\right] 
        =\begin{bmatrix}0\\t^3\\5t^4\end{bmatrix}+\begin{bmatrix}0\\1\\5\end{bmatrix}t^2+\begin{bmatrix}t\\0\\5\end{bmatrix}$$  has the form prescribed by \eq{piv}, while its coefficient matrix $$\cv(\bar \vv)=\left[\begin{array}{ccccc}
0 & 1 & 0 & 0 & 0 
\\
 0 & 0 & 1 & 1 & 0 
\\
 5 & 0 & 5 & 0 & 5 
\end{array}\right]$$
has the staircase shape described in \eqref{vbar}.  With $n=3, d=4$ and $k=2$, we have $n-(d-k-1)=2$ and so $\bv_{n-(d-k-1)}= \bv_{2}=t^3+t^2$. Thus \eqref{eq-rho2} becomes
$$\rho_2(\vv)=\frac 1 3\frac{{\rm coeff}(\bv_{2},t,2)}{{\rm coeff}(\bv_{2},t,3)}=\frac 1 3.$$ 
By looking at the entries of the matrix $\cv(\bar \vv)$, observe that this result is consistent  with~\eq{rho2v}, because we have $k<d-1$ and $b_{n-(d-k-1)}=1$. 

\item Now we have
\begin{align*}
\pi_2(\vv)= &\ \  \rho_2(\vv)^{-1}\cdot\vv \\
          = &\ \  \left[\begin{array}{c}
                        \left(t -\frac{1}{3}\right)^{4}+\left(t -\frac{1}{3}\right)^{3}+2 \left(t -\frac{1}{3}\right)^{2}+1 
                        \\
                        2 \left(t -\frac{1}{3}\right)^{4}+\left(t -\frac{1}{3}\right)^{3}+3 \left(t -\frac{1}{3}\right)^{2}+2 
                        \\
                        3 \left(t -\frac{1}{3}\right)^{4}+\left(t -\frac{1}{3}\right)^{3}+4 \left(t -\frac{1}{3}\right)^{2}+5 t +\frac{4}{3} 
                \end{array}\right]\\
           = &\ \  \left[\begin{array}{c}
                t^{4}-\frac{1}{3} t^{3}+\frac{5}{3} t^{2}-\frac{31}{27} t +\frac{97}{81} 
                \\
                2 t^{4}-\frac{5}{3} t^{3}+\frac{10}{3} t^{2}-\frac{53}{27} t +\frac{188}{81} 
                \\
                3 t^{4}-3 t^{3}+5 t^{2}+\frac{20}{9} t +\frac{16}{9} 
            \end{array}\right] \\
 \rho_1(\pi_2(\vv)) = &\ \  \left[\begin{array}{ccc}
                        -\frac{31}{27} & -\frac{1}{3} & \frac{1}{5} 
                        \\
                        -\frac{53}{27} & -\frac{5}{3} & \frac{2}{5} 
                        \\
                        \frac{20}{9} & -3 & \frac{3}{5} 
                \end{array}\right]
\end{align*}
                
Defining
$$\rho(\vv)=\big(\rho_1(\pi_2(\vv),\,\rho_2(\vv)\big)=\left(\left[\begin{array}{ccc}
                        -\frac{31}{27} & -\frac{1}{3} & \frac{1}{5} 
                        \\
                        -\frac{53}{27} & -\frac{5}{3} & \frac{2}{5} 
                        \\
                        \frac{20}{9} & -3 & \frac{3}{5} 
                \end{array}\right],\quad\frac 13 \right),$$

we observe that
$$\pi(\vv)=\rho(\vv)^{-1}\cdot \vv=\left[\begin{array}{c}
                        t -\frac{1}{3} 
                        \\
                        t^{3}-\frac{1}{27} 
                        \\
                        5 t^{4}+\frac{25}{3} t^{2}+\frac{325}{81} 
                \end{array}\right]              
                $$       
has  the coefficient  matrix
$$V(\pi(\vv))=\left[\begin{array}{ccccc}
-\frac{1}{3} & 1 & 0 & 0 & 0 
\\
 -\frac{1}{27} & 0 & 0 & 1 & 0 
\\
 \frac{325}{81} & 0 & \frac{25}{3} & 0 & 5 
\end{array}\right]
$$
in the shape prescribed by \eq{canonical}.      
\end{enumerate}
\qq \end{example}

The following example illustrates discontinuity, relative to the standard or to the Hausdorff topology, of the equivariant sections $\rho_1$ and $\rho_2$ and, therefore, of section $\rho$ defined in Theorem~\ref{thm-rho12}.
\begin{example}[discontinuity of sections]\label{ex-disc} \rm Consider a one parametric family of vectors $$\vv_c=\begin{bmatrix} t^4+5\\ t^3+7\\c\,t^2+t+4\end{bmatrix}, \quad c\in\R.$$
For $c\neq 0$, we have 
$$\rho_1(\vv_c)=\left[\begin{array}{ccc}
\frac{3}{2 c^{2}} & -\frac{2}{c} & -\frac{1}{c} 
\\
 -\frac{3}{2 c} & 1 & 0 
\\
 c  & 0 & 0 
\end{array}\right]\text{ and } \rho_2(\vv_c)=\frac 1{2\,c},
$$
while when $c=0$, we have
$$\rho_1(\vv_0)=\left[\begin{array}{ccc}
0 & 0 & -1
\\
0 & 1 & 0 
\\
1  & 0 & 0 
\end{array}\right]\text{ and } \rho_2(\vv_0)=0.
$$
\qq \end{example}

\noindent In our algorithm we will use the following corollary:

\begin{corollary}[equivariant matrix-completion]\label{cor-eq-mc}  Let $\mc\colon \rv\to  \K[t]^{n\times n} $ be a matrix-completion map,
           and the map  $\rho\colon\rv \to \slnk\times \K$ be  defined by \eq{eq-product-rho} in Theorem~\ref{thm-rho12}. 
          Then the map $\emc\colon \rv\to \K[t]^{n\times n}$ 
        defined by 
        \beq\label{eq-mcG1G2}\emc(\vv)=\rho(\vv)\cdot \mc\left(\rho(\vv)^{-1}\cdot\vv\right)\eeq
       is  an  $\slnk\times\K$-equivariant matrix-completion map.
  \end{corollary}

\section{Minimal-degree  matrix-completion ({\rm{\bf MMC}})} \label{sect-mmc}
The matrix-completion problem is well studied. In  Zhou and Labahn \cite{zhou-labahn-2014}, a non-minimal degree solution of a more general version of this problem is tackled: given $m<n$ vectors in $\pv$ find an $n\times n$ polynomial matrix with unit determinant.  A {non-minimal degree} matrix completion can be also obtained as the matrix inverse of a minimal multiplier, presented in  Beckermann,  Labahn, and Villard \cite{beckermann-labahn-villard-2006}, or, equivalently, as the matrix inverse of a constructive  solution to the Quillen-Suslin problem, as presented, for instance,  by  Fitchas and Galligo \cite{fitchas-galligo-1990}, Logar and Sturmfels \cite{logar-sturmfels-1992}, Caniglia,  Corti\~nas, Dan\'on,  Heintz, Krick, and  Solern\'o \cite{caniglia-1993}, Park and Woodburn \cite{park-woodburn-1995}, Lombardi and Yengui \cite{lombardi-yengui-2005}, Fabia\'nska and Quadrat \cite{fabianska-quadrat-2007},  Hong,  Hough,  Kogan and Li \cite{omf-2017}, Hough \cite{hough-2018}(see also Remark~\ref{rem-qs} and Example~\ref{ex-qs} of the current paper).  Subsequently, a  degree-reduction algorithm can be applied to minimize the degree. 

In Section~\ref{subs-mmc} we propose a novel minimal-degree matrix-completion algorithm, complementing a body of literature on this topic. In Section~\ref{subs-mind}, we  prove that the degree-minimality  property  of \emph{any}  matrix-completion algorithm  is preserved under equivariantization.
\subsection{{\rm{\bf MMC}} from a minimal \bez vector and its $\mu$-basis}\label{subs-mmc}

In this section, we show that by adjoining $\vv\in\pv$, such that $\gcd(\vv)=1$, with a $\mu$-basis of a \bez vector $\bb$ of $\vv$, one produces a matrix completion of $\vv$ of degree equal to $\deg\vv+\deg\bb$. If $\bb$ is of minimal degree then so is the matrix completion.    
We start by giving relevant definitions.
\begin{definition} [{B\'ezout} vector]\label{def-bez}A \emph{B\'ezout vector} of a
polynomial vector $\ww\in\pv$ is a polynomial  vector $\bb\in\pv$ such that the scalar  product 
$$\big<\ww,\bb\big>=\lambda\gcd(\ww),$$
where $\lambda\neq 0\in \K$ and  $\gcd(\ww)$ assumed to be monic. If $\lambda=1$, then a \bez vector is called \emph{normalized}\footnote{Any \bez vector can be normalized by dividing it  by $\lambda$.}. If the degree of $\bb$ is minimal among the degrees of all \bez vectors, then $\bb$ is called a \emph{minimal-degree \bez vector}, or just a \emph{minimal \bez vector}. 
\end{definition}
The existence of a \bez vector is a classical result, and algorithms for computing it are abundant. For completeness, in the Appendix,  we present Theorem~\ref{thm-bez}  that reduces computing a minimal \bez vector to computing the row-echelon form of a matrix over $\K$. This theorem and the resulting algorithm first appeared in \cite{omf-2017,hough-2018}. 
\begin{example}[B\'{e}zout vectors]\label{ex-bez}\rm For the vector  $$\vv=\begin{bmatrix} t^6+1\\t^3\\t\end{bmatrix}$$ in Example~\ref{ex-mc}  each of the following vectors
\beq\label{3bez} \bb_1=\begin{bmatrix} 1\\-t^3\\0\end{bmatrix}, \quad \bb_2=\begin{bmatrix} 1\\-t^3-1\\t^2\end{bmatrix}, \quad  \quad \bb_3=\begin{bmatrix} 1+t\\-t^3-t^4\\-1\end{bmatrix}
\eeq
 is a  normalized normalized B\'{e}zout vector with
 $$\deg\bb_1=\deg\bb_2=3 \text{ and }\deg\bb_3=4.$$
It is not difficult to observe that $\bb_1$ and $\bb_2$ have the minimal degree.
\qq \end{example}
\begin{definition}[outer product] \label{def-outer} For $n>1$, given  an  $(n-1)$-tuple of  polynomial vectors $\uu_1,\dots\uu_{n-1}\in\pv$, we define their \emph{outer product}, denoted  
$$\uu_1\wedge\dots\wedge\uu_{n-1},$$ to be a vector $\hh\in \pv$, constructed by first  forming an   $n\times(n-1)$-matrix $\pU=\begin{bmatrix}\uu_1&\cdots& \uu_{n-1}\end{bmatrix}$, and then defining the $\{i=1,\dots n\}$-th component of $\hh$  to be  $(-1)^{i+1} |\pU_{\hat i}|$, where $\pU_{\hat i}$ is the  $(n-1)\times(n-1)$-submatrix obtained by removing the $i$-th row from $\pU$.
\end{definition}
\begin{definition} [$\mu$-basis]\label{def-mu} Let $\ww\in \pv$, $n>1$, be such that $\gcd(\ww)=1$.  A \emph{$\mu$-basis} of $\ww$ is an  $(n-1)$-tuple of polynomial  vectors $\uu_1,\dots,\uu_{n-1}\in\pv$ such that: 
   $$\sum_{i=1}^{n-1}\deg \uu_i =\deg \ww \text{ and } \uu_1\wedge\dots\wedge\uu_{n-1}=\lambda \ww .$$ for some $\lambda\neq 0\in \K$. If $\lambda=1$, then a \mub  is called \emph{normalized}\footnote{Any \mub $\uu_1,\,\dots,\,\uu_{n-1}$  can be normalized  by dividing any one of the $\uu_i$'s (say the last one)  by $\lambda$.}.
\end{definition}

\begin{remark}\label{rem-mu}\rm It follows  immediately from the definition that  every element $\uu_i$ of a \mub  of $\ww$ belongs to the syzygy module of $\ww$, where
$\syz(\ww)=\{\hh\in \pv\,|\,\big<\ww,\hh\big>=0 \}.$    
It is well known (and not difficult to show) that $\syz(\ww)$ is a free module and a \mub is its column-wise optimal degree basis. We refer the reader to \cite{song-goldman-2009} (in particular, to  Definition 1 and Theorem 2 therein) for  the discussion of several equivalent  definitions of a $\mu$-basis. We chose the most convenient definition for our purposes. 
\qq \end{remark}
\begin{example}[$\mu$-basis]\label{ex-mub}\rm For the vector $$\vv=\begin{bmatrix} t^6+1\\t^3\\t\end{bmatrix}$$ appearing  in Example~\ref{ex-mc} and~\ref{ex-bez} and vectors
$$\uu_1=\begin{bmatrix} 0\\-1\\t^2\end{bmatrix}\text{ and } \uu_2=\begin{bmatrix} t\\-t^4\\-1\end{bmatrix},$$ 
it is easy to check that 
$$\vv=\uu_1\wedge\uu_2\text{ and } \deg\uu_1+\deg\uu_2=\deg \vv.$$
Therefore, $\uu_1$ and $\uu_2$ satisfy Definition~\ref{def-mu} of a normalized $\mu$-basis of $\vv$. As observed in Remark~\ref{rem-mu}, $\uu_1$ and $\uu_2$ is a basis of $\syz(\vv)$. As a side remark, note that B\'{e}zout vectors $\bb_2$ and $\bb_3$, appearing in Example~\ref{ex-bez}, were obtained by adding, respectively, $\uu_1$ and $\uu_2$ to $\bb_1$.
\qq \end{example}
The  notion of a $\mu$-basis, although  less standard than of B\'{e}zout vector, has a long history of applications in geometric modeling, originating with works by Sederberg and Chen \cite{sederberg-chen-1995},   Cox,  Sederberg and Chen \cite{cox-sederberg-chen-1998}. For further development of   the theory  and computation of $\mu$-bases see, for instance,  \cite{zheng-sederberg-2001, chen-wang-2002, andrea-2004, chen-cox-liu-2005, song-goldman-2009, jia-goldman-2009,tesemma-wang-2014, hhk2017, jia-shi-chen-2018})
 The problem of computing a $\mu$-basis  can be  also viewed as a particular case of the problem of computing optimal-degree kernels of~$m\times n$ polynomial matrices of rank $m$ (see for instance Beelen \cite{beelen1987},   Antoniou,  Vardulakis, and Vologiannidis \cite{antoniou2005},  Zhou, Labahn, and Storjohann \cite{zhou-2012} and references therein). For completeness, in the Appendix,  we present Theorem~\ref{thm-mu}  that reduces computing a \mub  to computing the row-echelon form of a matrix over $\K$. This theorem and the resulting algorithm first appeared~in~\cite{hhk2017}.

The next theorem shows how one can use a \bez vector of $\vv$ and a $\mu$-basis of \emph{the \bez vector} to construct a \emph{minimal-degree} matrix-completion  of $\vv$. 
\begin{theorem}[minimal-degree matrix-completion]\label{thm-mc} Let $\vv\in\mathbb{K}[t]^{n}$ be a polynomial vector
such that $\gcd({\vv})=1$. Assume  $\bb$ is a normalized {B\'{e}zout} vector of
$\vv$ 
and  $\left\{
\uu_{1},\dots,\uu_{n-1}\right\}  $ is a  normalized $\mu$-basis of
$\bb$,
 then: 
\begin{enumerate}
\item The matrix  $\pM=\begin{bmatrix}\vv&\uu_1&\dots&\uu_{n-1}\end{bmatrix}$ is a matrix-completion of $\vv$.
\item $\deg \pM =\deg\vv+\deg\bb$.

\item If $\bb$ is a minimal-degree {B\'{e}zout }vector, then $\pM$  is a minimal-degree matrix-completion of~$\vv$.
\end{enumerate}
\end{theorem}
\begin{proof}
Using the definitions of a \bez vector and a $\mu$-basis, we have:
\begin{enumerate}
\item $|\pM|=\left|\begin{matrix}\vv&\uu_1&\dots&\uu_{n-1}\end{matrix}\right|=\big<\vv, \uu_1\wedge\dots\wedge\uu_{n-1}\big>=\big<\vv,\bb\big>=1 $.
 \item $\deg \pM =\deg\vv+\sum_{i=1}^{n-1}\deg(\uu_i) =\deg\vv+\deg \bb$.
\item Assume $\tilde \pM=\begin{bmatrix}\vv&\ww_1&\cdots& \ww_{n-1}\end{bmatrix}$  is any matrix-completion of $\vv$. Let $\ba=\ww_1\wedge\dots\wedge\ww_{n-1}$. Since $|\tilde \pM|=\big<\vv,\ba\big>=1$, then $\ba$ is a \bez vector of $\vv$.  The definition of the outer product implies that $\deg \ba\leq \deg \ww_1+\dots+\deg\ww_{n-1}$ and  so
 $$\deg \tilde \pM =\deg\vv+\sum_{i=1}^{n-1}\deg(\ww_i) \geq\deg\vv+\deg \ba\geq \deg\vv+\deg \bb,$$
 where the last inequality is due to degree-minimality of $\bb$. From Part 2 it follows that $\pM$ is of minimal degree. 
\end{enumerate}
\end{proof}
\begin{remark}\label{rem-norm}\rm 
If,  in the above theorem, we omit the requirement of a  \bez vector and its \mub being \emph{normalized}, then the resulting matrix $\pM$ will have a \emph{non-zero constant determinant}   $\lambda\in\K$. In this case, by dividing the last column of  $\pM$ by $\lambda$ we obtain a unit determinant matrix-completion of $\vv$.  This is the approach we adopt in  Algorithm~\ref{alg-mmc} and its implementation.
\qq \end{remark}
We have an immediate but a very useful corollary of Theorem~\ref{thm-mc}.
\begin{corollary}\label{cor-mmc} Let   $\pM$ be a matrix-completion of a polynomial vector $\vv$.  The degree of $\pM$ is minimal if and only if $\deg \pM=\deg \vv+ \beta$, where $\beta$ is the  minimal degree of \bez vectors of~$\vv$.
\end{corollary}
\begin{example}\label{ex-mmc} \rm Returning to $$\vv=\begin{bmatrix} t^6+1\\t^3\\t\end{bmatrix}$$ from Example~\ref{ex-bez} Let us choose its normalized minimal   \bez   vector $$\bb=\bb_1=\begin{bmatrix} 1\\-t^3\\0\end{bmatrix}.$$ A normalized  $\mu$-basis of $\bb$ consists of
$$\uu_1= \begin{bmatrix} 0\\0\\-1\end{bmatrix}\text{ and } \uu_2= \begin{bmatrix} t^3\\1\\0\end{bmatrix}.$$
(\emph{Note that in Example~\ref{ex-mub}, we  presented  a $\mu$-basis of $\vv$, not of $\bb$}.) Observe that matrix-completion $\pM_1$ of $\vv$ presented in Example~\ref{ex-mc} equals to $\begin{bmatrix}\vv&\uu_1&\uu_2\end{bmatrix}$. From Theorem~\ref{thm-mc}, we know that $\pM_1$ is of minimal degree. Matrix completions $\pM_2$ and $\pM_3$ of Example~\ref{ex-mc} are obtained by adjoining $\vv$ with normalized $\mu$-bases of  \bez vectors $\bb_2$ and $\bb_3$ appearing in Example~\ref{ex-bez}.
 
\qq \end{example}
In Remark~\ref{rem-qs} below we discuss the relationship between the matrix-completion problem and the   \QS  problem. The latter can be solved using \bez vector and a $\mu$-basis of \emph{$\vv$ itself}.
 
\begin{remark}\label{rem-qs}  \rm A matrix $\pQ\in\K[t]^{n\times n}$ is called a \QS matrix for  $\vv\in\pv$ if
$$\vv^T\pQ=\begin{bmatrix}1 & 0&\dots&0\end{bmatrix}.$$  It is not difficult to show that $\pQ$ is  a \QS matrix for $\vv$ if and only if 
$\pQ^{-T}$ is a matrix-completion of $\vv$. As discussed in \cite{omf-2017, hough-2018}, assuming $\gcd(\vv)=1$, a \QS matrix~for $\vv$ can be constructed as follows.\footnote{In \cite{omf-2017, hough-2018}, a \QS matrix is called a moving frame of $\vv$. The definition of a moving frame given in the current paper is more closely related to the classical differential geometric frames.}

Let  $\bb$ be a normalized minimal {B\'{e}zout} vector and  $\left\{
\uu_{1},\dots,\uu_{n-1}\right\} $ be a  normalized $\mu$-basis of
$\vv$\footnote{Note the difference with Theorem~\ref{thm-mc}, where~$\left\{
\uu_{1},\dots,\uu_{n-1}\right\}  $ is a $\mu$-basis of $\bb$.} ordered so that $\deg \uu_1\leq\cdots\leq \uu_{n-1}$, then
\beq\label{QSm}\pQ=\begin{bmatrix}\bb&\uu_1&\cdots& \uu_{n-1}\end{bmatrix}\eeq
is a \QS matrix of $\vv$.  
Moreover, $\pQ$ has column-wise optimal degree.  In other words, if $\pP$ is another \QS matrix of $\vv$, whose last $n-1$ columns  are ordered so that their degrees are non-decreasing, then
$$\deg \pP_{*1}\geq\deg \bb,\quad \deg \pP_{*2}\geq\deg \uu_1,\,\dots, \,\deg\pP_{*n}\geq\deg \uu_{n-1}.$$
We note that  column-wise degree optimality is a stronger condition than degree-minimality. 
Due to a close relationship between \QS and   matrix-completion matrices, we would like to underscore \emph{important differences}:
\begin{itemize}
\item As illustrated  in Example~\ref{ex-mc} above,  column-wise degree optimality is not achievable for solutions of the matrix-completion problem.
\item 
As illustrated  in Example~\ref{ex-qs} below:
\begin{itemize}
\item  column-wise degree optimality of a \QS matrix  $\pQ$ does not imply degree-minimality of the corresponding matrix-completion $\pM=\pQ^{-T}$.  
\item degree-minimality of a matrix-completion $\pM$ does not imply column-wise degree optimality of  the corresponding  \QS matrix  $\pQ=\pM^{-T}$
\end{itemize}
\end{itemize}

\qq \end{remark}
\begin{example}\label{ex-qs}\rm For $$\vv=\begin{bmatrix} t^6+1\\t^3\\t\end{bmatrix},$$ we can use a  minimal \bez  $\bb=\bb_1$ and  a \mub $[\uu_1,\uu_2]$ of $\vv$,  presented in  Examples~\ref{ex-bez} and~\ref{ex-mub}, respectively, to produce a \QS matrix
$$\pQ=\begin{bmatrix}\bb&\uu_1&\uu_{3}\end{bmatrix}=\left[\begin{array}{ccc}
1 & 0 & t  
\\
 -t^{3} & -1 & -t^{4} 
\\
 0 & t^{2} & -1 
\end{array}\right].
$$
The matrix $$\pQ^{-T}=\left[\begin{array}{ccc}
t^{6}+1 & -t^{3} & -t^{5} 
\\
 t^{3} & -1 & -t^{2} 
\\
 t  & 0 & -1 
\end{array}\right]
$$
is a matrix completion of $\vv$ of degree 14, which  is \emph{not minimal}. On the other hand, taking the inverse-transpose of a minimal-degree matrix-completion completion $\pM_1$, presented in Example~\ref{ex-mc}, we obtain a \QS matrix 
$$(\pM_1)^{-T}=\left[\begin{array}{ccc}
1 & t  & -t^{3} 
\\
 -t^{3} & -t^{4} & t^{6}+1 
\\
 0 & -1 & 0 
\end{array}\right]
$$
of non-optimal degree.
\qq \end{example}

\subsection{Minimality degree preservation}\label{subs-mind}
In this section, we show that minimal-degree property of \emph{any} matrix-completion map is preserved   by the equivariantization process described in Theorem~\ref{thm-eq-mc}.
We will need a few preliminary results.

 
 \begin{lemma}\label{lem-deg-inv} For $\vv\in\pv$ and $(L,s)\in \slnk\times\K$,
 $$\deg \vv=\deg  \left[(L,s)\cdot \vv\right],$$
 where, as we defined before, $(L,s)\cdot \vv(t)=L\vv(t+s)$.
 \end{lemma}
 \begin{proof} Observe that $\deg \vv(t)=\deg \vv(t+s)$ because a shift of the parameter does not change the degree of a polynomial. Let  $$\vv=\begin{bmatrix} \vv_1\\ \vdots \\ \vv_n\end{bmatrix},$$ and $$\deg \vv=\max_{i=1\dots n}\deg \vv_i=d,\;\;\;\; \text{and}\;\;\;\; {\rm coeff}(\vv_i,d)=c_{id}.$$  
 
\noindent Since $ L\vv=\vv_1L_{*1}+\dots+\vv_nL_{*n}$, we have $\deg (L\vv)\leq \deg \vv=d$, and $\deg L\vv<d$ if and only if
 \beq\label{deg-loss} c_{1d}L_{*1}+\dots+c_{nd}L_{*n}=0.\eeq
Since at least one of the coefficients $c_{id}$ is non-zero, \eq{deg-loss} implies  linear dependence of the columns of $L$, which contradicts our assumption that $L\in \slnk$ and so $|L|=1$. Thus $\deg L\vv =\deg\vv $. 
  \end{proof}

 \begin{corollary}\label{cor-deg-inv} For $\pW(t)\in\K[t]^{n\times l}$ and $(L,s)\in \slnk\times\K$,
 $$\deg \pW=\deg  \left[(L,s)\cdot \pW\right],$$
 where  $(L,s)\cdot \pW(t)=L\pW(t+s)$.
 \end{corollary}
 \begin{proof} First, note that the  transformation induced by $(L,s)$ is column-wise:
  $$\big[(L,s)\cdot \pW\big]_{*i}=(L,s)\cdot \pW_{*i},\qquad  i=1,\dots l.$$  
  Then, from Lemma~\ref{lem-deg-inv}, $\deg\big[ (L,s)\cdot \pW\big]_{*i}=\deg \pW_{*i}$.
    Since the degree of a  polynomial matrix is defined as the sum of the  degrees of its columns, the conclusion follows.

  \end{proof}
  \begin{lemma}\label{lem-bez-deg-inv}   For $\vv\in\pv$ such that $\gcd(\vv)=1$ and $(L,s)\in \slnk\times\K$, let
  $\bv=(L,s)\cdot \vv$, and let $\bb$ and $\bar \bb$ be minimal degree \bez vectors of $\vv$ and $\bv$, respectively.
  Then $\gcd(\bv)=1$ and  $$\deg\bb=\deg\bar \bb.$$
 \end{lemma}
\begin{proof} Since $\big<\bb(t),\vv(t)\big>=1$, 
we have  $$\big< L^{-T}\bb(t+s),L\vv(t+s)\big>=\big<\tilde\bb,\bv\big>=1,$$ where $$\tilde \bb=(L^{-T},s)\cdot \bb=L^{-T}\bb(t+s).$$ Therefore, $\gcd(\bv)=1$ and $\tilde \bb$ is a \bez vector of $\bv$. From Lemma~\ref{lem-deg-inv}, $\deg \tilde b=\deg b$. 

Assume $\tilde\bb$ is not a \emph{minimal-degree} \bez vector of $\bv$, i.e. $\deg \tilde\bb>\deg \bar\bb$.  Applying the above argument to $\bar \bb$, we see that $\hat\bb=(L^{T},-s)\cdot\bar\bb$ is a \bez vector of $\vv$ such that $\deg\hat\bb=\deg\bar \bb$. From the previous paragraph we have, however, that $\deg\bar \bb<\deg\tilde\bb=\deg \bb$.  Thus $\deg \hat \bb<\deg \bb$, which contradicts the assumption that $\bb$ is of minimal degree. We are forced to conclude that~$\tilde \bb$ is a minimal degree \bez vector of $\bv$ and so
 $\deg\bar\bb=\deg \tilde \bb=\deg \bb$. 
\end{proof}
\begin{proposition}\label{prop-min-deg-inv}
Let $\mc\colon \rv\to  \K[t]^{n\times n}$  be a minimal degree matrix-completion map.  For any  $(L,s)\in \slnk\times \K$ the map $\emc\colon \rv\rightarrow \K[t]^{n\times n}$ 
defined by 
\beq\label{act-mdmc}\emc(\vv)=(L,s)\cdot \mc\left((L,s)^{-1}\cdot\vv\right)\eeq
 is  a minimal degree matrix-completion map.
\end{proposition}
\begin{proof} 
\begin{enumerate}
\item By Lemma~\ref{lem-mc}, $\emc$  is a matrix-completion map.
\item  By Corollary~\ref{cor-deg-inv}, $\deg\emc(\vv)=\deg \mc\left((L,s)^{-1}\cdot\vv\right)$.
\item By Theorem~\ref{thm-mc}  $\deg \mc\left((L,s)^{-1}\cdot\vv\right)=\deg \bv+\deg \bar\bb$, where $\bar v=(L,s)^{-1}\cdot\vv$ and $\bar \bb$ is a minimal  degree \bez vector of $\bv$.
\item By Lemmas~\ref{lem-deg-inv} and~\ref{lem-bez-deg-inv}, $\deg\bv=\deg\vv$ and $\deg\bar\bb=\deg\bb$, where  $\bb$ is a minimal-degree \bez vector of $\vv$.
\end{enumerate}
Combining the above facts, we conclude that $\emc$ is a matrix  completion map, such that, for every $\vv\in\rv$,  $\deg\emc(\vv)=\deg \vv+\deg \bb$, where  $\bb$ is a minimal  degree \bez vector of $\vv$.  Corollary \ref{cor-mmc}, implies that $\emc$ is a minimal degree matrix-completion map.
\end{proof}

\noindent We immediately have the following result:
\begin{theorem}[equivariant minimal-degree matrix-completion map]\label{thm-min-deg-inv}\hfill

\noindent  Let \begin{itemize}
 \item $\mc\colon \rv\to  \K[t]^{n\times n} $ be a minimal-degree matrix-completion map,
     \item    $\rho\colon\rv \to \slnk\times \K$  be any equivariant section (for instance, the map defined by \eq{eq-product-rho} in Theorem~\ref{thm-rho12}).
     \end{itemize}
    Then the map $\emc\colon \rv\rightarrow \K[t]^{n\times n}$ 
defined by 
\beq\label{eqv-mc}\emc(\vv)=\rho(\vv)\cdot \mc\left(\rho(\vv)^{-1}\cdot\vv\right)\eeq
 is  an equivariant  minimal-degree matrix-completion map.
\end{theorem}

\begin{proof}

The fact that the map  $\emc\colon \rv\rightarrow \K[t]^{n\times n}$ is   an equivariant \emph{matrix completion map}  follows  from Theorem~\ref{thm-eq-mc}. To show that it is of \emph{minimal degree} map we can use the same argument as in Proposition~\ref{prop-min-deg-inv} with $L$ replaced by $L(\vv)$ and $s$ replaced with $s(\vv)$, where $\rho(\vv)=\left(L(\vv),s(\vv)\right)$.

 \end{proof}
\begin{remark}\rm Theorem~\ref{thm-min-deg-inv} tells us that if $\mc$ is of \emph{minimal degree} and $\emc$ is its equivariantization then
\beq\label{eq-deg-mc}\deg\mc(\vv)=\deg\emc(\vv) \text { for all } \vv\in\rv.\eeq 
It is worthwhile to note that condition~\eq{eq-deg-mc} does not hold if  $\mc$ \emph{is not a minimal-degree} matrix-completion map. In other words, equivariantization of an arbitrary matrix-completion map does not have to preserve the degree of the output.

To illustrate this, we define a non-minimal matrix completion map $\mathcal N$ as follows. For~$\vv \in \rv$, let $\ww_1,\dots, \ww_{n-1}$ is a \mub of $\vv$ ordered so that $\deg \ww_1\leq\cdots\leq \ww_{n-1}.$   Let $\bb$ be a minimal \bez vector of $\vv$. Then  $\hat\bb=\bb+t^{\deg(\vv_1)}\ww_{n-1}$, where $\vv_1$ is the first component of vector $\vv$, is also a \bez vector of $\vv$. However, since $\deg \bb < \deg \ww_{n-1}$ (see \cite[Theorem 2]{omf-2017}), we have $\deg\hat\bb>\deg \bb$, and so this is  a \emph{non-minimal}  \bez vector. Let $\uu_1,\dots, \uu_{n-1}$ be a \mub of $\hat\bb$. Then by Theorem~\ref{thm-mc} 
$$\mathcal N(\vv)=\begin{bmatrix}\vv&\uu_1&\dots&\uu_{n-1}\end{bmatrix} $$
is a non-minimal matrix completion of $\vv$ of degree $\deg\vv+\deg\hat\bb=\deg\vv+\deg\vv_1+\deg\ww_{n-1} $.  

We use the above construction along with an implementation of $\mu$-basis and minimal-degree \bez vector algorithms  described in the Appendix, to obtain a concrete example of  a \emph{non-minimal} matrix completion algorithm also implemented in \cite{code}. Let $\widetilde{\mathcal N}$ be an equivariantization of $\mathcal N$ obtained as prescribed by Theorem~\ref{thm-min-deg-inv}, with $\rho\colon\rv \to \slnk\times \K$ given  by \eq{eq-product-rho} in Theorem~\ref{thm-rho12}.

For $\vv=[t^3,t,1]^T$ we compute that $\rho(\vv)=(L,0)$, where $L= \left[\begin{array}{ccc}
0 & 0 & -1 
\\
 0 & 1 & 0 
\\
 1 & 0 & 0 
\end{array}\right]$.
Furthermore, for $\vv$ we have:
$$\mathcal N(\vv)=\left[\begin{array}{ccc}
t^{3} & t^{2} & -1 
\\
 t  & 1 & 0 
\\
 1 & 0 & -t^{3} 
\end{array}\right] \text{ while } \widetilde{\mathcal N}(\vv)= L\mathcal N(L^{-1}\vv)=\left[\begin{array}{ccc}
t^{3} & -1 & 0 
\\
 t  & 0 & -1 
\\
 1 & -1 & t^{2} 
\end{array}\right]
$$
and observe that
$\deg \mathcal N(\vv) > \deg \widetilde{\mathcal N}(\vv)$. \hfill \qq
\end{remark}
\section{Equi-affine minimal-degree moving frames}\label{sect-mdeamf}
Finally, we integrate  the results of previous sections into one
theorem  that describes an equi-affine minimal-degree moving frame map. Then we   describe the map   as a self-contained algorithm for the readers who would like
to implement it without knowing all the underlying theory. Our own code is posted \cite{code}. 
\begin{theorem}[equi-affine minimal-degree moving frame map]\label{thm-mdeamf}\hfill

\noindent  Let
        \begin{itemize}
                \item $\mc\colon \rv\to  \K[t]^{n\times n} $ be a minimal-degree matrix-completion map,
                \item  $\rho\colon\rv \to \slnk\times \K$   be any equivariant section (for instance, the map defined by \eq{eq-product-rho} in Theorem~\ref{thm-rho12}).        \end{itemize}
        Then the map $\mfm\colon \rc\to \K[t]^{n\times n}$ 
        defined by 
        \beq\label{eq-mdeamf}\mfm(\cc)=\rho(\cc')\cdot \mc\left(\rho(\cc')^{-1}\cdot\cc'\right)\eeq
        is an equi-affine minimal-degree moving frame map.
        
\end{theorem}
\begin{proof}
\begin{enumerate}
\item By Theorem~\ref{thm-min-deg-inv}, $\emc\colon\rv\to \pmat$, defined by $\emc(\vv)=\rho(\vv)\cdot \mc\left(\rho(\vv)^{-1}\cdot\vv\right)$ is an equivariant  minimal-degree matrix-completion map.

\item By Theorem ~\ref{thm-trivial},    $\mfm\colon \rc\to \K[t]^{n\times n}$ is a minimal degree equi-affine moving frame map. 

 \end{enumerate}

 \end{proof}
 
 \medskip
 
\noindent {The above theorem leads to the following algorithm, which we describe below as a sequence of procedures, in  the top-down order.}

\begin{algorithm}
        [Equi-Affine Moving Frame Map]\ $\pF\leftarrow EAMFM\left(  \cc\right)  $
        
        \begin{description}
                \item[In:\;\;\;] $\cc \in \mathcal{C}$
                
                \item[Out:] $\pF \in \K[t]^{n\times n}$
        \end{description}
        
        \begin{enumerate}
                \item $\vv\leftarrow \cc^{\prime}$
                
                \item $\pF\leftarrow EMCM\left(  \vv\right)  $
        \end{enumerate}
\end{algorithm}

\begin{example} $EAMFM$\  	

	\begin{enumerate}
    	\item[] $\cc = \left[\begin{array}{c}
    		\frac{1}{5} t^{5}+\frac{1}{4} t^{4}+\frac{2}{3} t^{3}+t  
    		\\
    		\frac{2}{5} t^{5}+\frac{1}{4} t^{4}+t^{3}+2 t  
    		\\
    		\frac{3}{5} t^{5}+\frac{1}{4} t^{4}+\frac{4}{3} t^{3}+\frac{5}{2} t^{2}+3 t  
    	\end{array}\right]
    	$
		\item $\vv\leftarrow \left[\begin{array}{c}
			t^{4}+t^{3}+2 t^{2}+1 
			\\
			2 t^{4}+t^{3}+3 t^{2}+2 
			\\
			3 t^{4}+t^{3}+4 t^{2}+5 t +3 
		\end{array}\right]
		$
		
		\item $\pF\leftarrow \left[\begin{array}{ccc}
			t^{4}+t^{3}+2 t^{2}+1 & 0 & \frac{16 t}{27} 
			\\
			2 t^{4}+t^{3}+3 t^{2}+2 & \frac{27}{40} & -\frac{22 t}{27}-\frac{34}{27} 
			\\
			3 t^{4}+t^{3}+4 t^{2}+5 t +3 & \frac{243}{80} & -\frac{65 t}{9}-\frac{113}{27} 
		\end{array}\right]
		$
		\end{enumerate}
	
\end{example}

\begin{algorithm}
        [Equivariant Matrix-Completion Map]\ $\pM\leftarrow EMCM\left( \vv\right)  $
        
        \begin{description}
                \item[In:\;\;\;] $\vv \in \rv$
                
                \item[Out:] $\pF \in \K[t]^{n\times n}$
        \end{description}
        
        \begin{enumerate}
                \item $L,s\leftarrow ES\left(  \vv\right)  $
                
                \item $\bar{\vv}\leftarrow L^{-1}\cdot s^{-1}\cdot{\vv}  $
                
                \item $\overline{\pM}\leftarrow MMC(\bar{\vv}) $
                
                \item $\pM\leftarrow s\cdot L\cdot\overline{\pM} $
        \end{enumerate}
\end{algorithm}

\begin{example} $EMCM$ \

	\begin{enumerate}
		\item[] $\vv = \left[\begin{array}{c}
			t^{4}+t^{3}+2 t^{2}+1 
			\\
			2 t^{4}+t^{3}+3 t^{2}+2 
			\\
			3 t^{4}+t^{3}+4 t^{2}+5 t +3 
		\end{array}\right]
		$

		\item $L \leftarrow 
		\left[\begin{array}{ccc}
			-\frac{31}{27} & -\frac{1}{3} & \frac{1}{5} 
			\\
			-\frac{53}{27} & -\frac{5}{3} & \frac{2}{5} 
			\\
			\frac{20}{9} & -3 & \frac{3}{5} 
		\end{array}\right]
		$
		 
		\item[] $s \leftarrow \frac{1}{3}$
		
		\item $\bar{\vv}\leftarrow \left[\begin{array}{c}
			t -\frac{1}{3} 
			\\
			t^{3}-\frac{1}{27} 
			\\
			5 t^{4}+\frac{25}{3} t^{2}+\frac{325}{81} 
		\end{array}\right]		
		$
		
		\item $\overline{\pM} \leftarrow \left[\begin{array}{ccc}
			t -\frac{1}{3} & \frac{27}{80} & -t  
			\\
			t^{3}-\frac{1}{27} & -\frac{9}{16} & \frac{5 t}{3}+\frac{16}{27} 
			\\
			5 t^{4}+\frac{25}{3} t^{2}+\frac{325}{81} & 1 & 0 
		\end{array}\right]
		 $
		
		\item $\pM\leftarrow \left[\begin{array}{ccc}
			t^{4}+t^{3}+2 t^{2}+1 & 0 & \frac{16 t}{27} 
			\\
			2 t^{4}+t^{3}+3 t^{2}+2 & \frac{27}{40} & -\frac{22 t}{27}-\frac{34}{27} 
			\\
			3 t^{4}+t^{3}+4 t^{2}+5 t +3 & \frac{243}{80} & -\frac{65 t}{9}-\frac{113}{27} 
		\end{array}\right]
		$
	\end{enumerate}
\end{example}

\begin{algorithm}
        [Equivariant Section]$(L,s)\leftarrow ES\left(  \vv\right)  $
        
        \begin{description}
                \item[In:\;\;\;] $\vv \in  \rv$
                
                \item[Out:] $(L,s) \in \K[t]^{n\times n} \times \K$             
        \end{description}
        
        \begin{enumerate}
                \item $s \leftarrow  ES2(\vv) $ 

                \item $\bar{\vv} \leftarrow s^{-1}\cdot\vv $ 

                \item $L \leftarrow  ES1(\bar{\vv})$
        \end{enumerate}
\end{algorithm}

\begin{example} $ES$\ 
	\begin{enumerate}
		\item[] $\vv=\left[\begin{array}{c}
			t^{4}+t^{3}+2 t^{2}+1 
			\\
			2 t^{4}+t^{3}+3 t^{2}+2 
			\\
			3 t^{4}+t^{3}+4 t^{2}+5 t +3 
		\end{array}\right]
		$

		\item $s \leftarrow  \frac{1}{3}$ 
		
		\item $\bar{\vv} \leftarrow \left[\begin{array}{c}
			\left(t -\frac{1}{3}\right)^{4}+\left(t -\frac{1}{3}\right)^{3}+2 \left(t -\frac{1}{3}\right)^{2}+1 
			\\
			2 \left(t -\frac{1}{3}\right)^{4}+\left(t -\frac{1}{3}\right)^{3}+3 \left(t -\frac{1}{3}\right)^{2}+2 
			\\
			3 \left(t -\frac{1}{3}\right)^{4}+\left(t -\frac{1}{3}\right)^{3}+4 \left(t -\frac{1}{3}\right)^{2}+5 t +\frac{4}{3} 
		\end{array}\right]
	    =\left[\begin{array}{c}
	    	t^{4}-\frac{1}{3} t^{3}+\frac{5}{3} t^{2}-\frac{31}{27} t +\frac{97}{81} 
	    	\\
	    	2 t^{4}-\frac{5}{3} t^{3}+\frac{10}{3} t^{2}-\frac{53}{27} t +\frac{188}{81} 
	    	\\
	    	3 t^{4}-3 t^{3}+5 t^{2}+\frac{20}{9} t +\frac{16}{9} 
	    \end{array}\right]
		$

		\item $L \leftarrow  \left[\begin{array}{ccc}
			-\frac{31}{27} & -\frac{1}{3} & \frac{1}{5} 
			\\
			-\frac{53}{27} & -\frac{5}{3} & \frac{2}{5} 
			\\
			\frac{20}{9} & -3 & \frac{3}{5} 
		\end{array}\right]
		$
	\end{enumerate}
\end{example}

\begin{algorithm}[Equivariant Section 1]
$L\leftarrow ES1(\vv) $

\begin{description}
        \item[In:] $\vv \in  \rv$
        
        \item[Out:] $L \in \K[t]^{n\times n} $
\end{description}

\begin{enumerate}
        \item $d\leftarrow\deg\left( \vv\right) $
        
        \item $\cv\leftarrow $ the $n \times (d+1)$ matrix such that $\vv=\cv\left[ 
        \begin{array}{c}
                t^{0} \\ 
                \vdots \\ 
                t^{d}%
        \end{array}
        \right] $
        
        \item $\bar{\cv}\leftarrow$ the right-most full rank submatrix of $\cv$
        
        \item $L\leftarrow$ the matrix obtained from $\overline{\cv}$ by dividing its last column by $|\bar{\cv}|$
\end{enumerate}
\end{algorithm}

\begin{example} $ES1$\ 
\begin{enumerate}
	\item[] $\vv =\left[\begin{array}{c}
		t^{4}-\frac{1}{3} t^{3}+\frac{5}{3} t^{2}-\frac{31}{27} t +\frac{97}{81} 
		\\
		2 t^{4}-\frac{5}{3} t^{3}+\frac{10}{3} t^{2}-\frac{53}{27} t +\frac{188}{81} 
		\\
		3 t^{4}-3 t^{3}+5 t^{2}+\frac{20}{9} t +\frac{16}{9} 
	\end{array}\right]
	$
	\item $d\leftarrow 4$
	
	\item $\cv\leftarrow \left[\begin{array}{ccccc}
		\frac{97}{81} & -\frac{31}{27} & \frac{5}{3} & -\frac{1}{3} & 1 
		\\
		\frac{188}{81} & -\frac{53}{27} & \frac{10}{3} & -\frac{5}{3} & 2 
		\\
		\frac{16}{9} & \frac{20}{9} & 5 & -3 & 3 
	\end{array}\right]
	$
	
	\item $\bar{\cv}\leftarrow \left[\begin{array}{ccc}
		-\frac{31}{27} & -\frac{1}{3} & 1 
		\\
		-\frac{53}{27} & -\frac{5}{3} & 2 
		\\
		\frac{20}{9} & -3 & 3 
	\end{array}\right]
	$
	
	\item $L\leftarrow \left[\begin{array}{ccc}
		-\frac{31}{27} & -\frac{1}{3} & \frac{1}{5} 
		\\
		-\frac{53}{27} & -\frac{5}{3} & \frac{2}{5} 
		\\
		\frac{20}{9} & -3 & \frac{3}{5} 
	\end{array}\right]
	$
\end{enumerate}
\end{example}

\begin{algorithm}[Equivariant Section 2]
$s\leftarrow ES2 \left( \vv\right) $

\begin{description}
        \item[In:] $\vv\in  \rv$
        
        \item[Out:] $s \in \K$
\end{description}

\begin{enumerate}
        \item $d\leftarrow\deg\left( \vv\right) $

    \item $\cv\leftarrow $ the $n \times (d+1)$ matrix such that $\vv=\cv\left[ 
        \begin{array}{c}
                t^{0} \\ 
                \vdots \\ 
                t^{d}%
        \end{array}
        \right] $
        
        \item $\bar{\cv}\leftarrow$ the right-most full rank submatrix of $\cv$
        
        \item $L\leftarrow$ the matrix obtained from $\bar{\cv}$ by dividing its last column by $|\bar{\cv}|$
        
        \item $k\leftarrow$ the index of the right most column of $\cv$
        not included in $\overline{\cv}$  
        
        \hspace{2em} (where the columns are indexed from $0$ to $d$)

        \item $\bar{\vv}\leftarrow L^{-1}\cdot \vv$
        
        \item $s\leftarrow \frac{\;\;\;\;\;\;\mathrm{coeff}(\bar{\vv}_{n-(d-k-1)},t,k)}{(k+1)\;\mathrm{coeff}(\bar{\vv}_{n-(d-k-1)},t,k+1)}$
\end{enumerate}
\end{algorithm}

\begin{example} $ES2$\ 
\begin{enumerate}
	\item[] $\vv = \left[\begin{array}{c}
		t^{4}+t^{3}+2 t^{2}+1 
		\\
		2 t^{4}+t^{3}+3 t^{2}+2 
		\\
		3 t^{4}+t^{3}+4 t^{2}+5 t +3 
	\end{array}\right]
	$
	\item $d\leftarrow 4$

    \item $\cv\leftarrow \left[\begin{array}{ccccc}
    	1 & 0 & 2 & 1 & 1 
    	\\
    	2 & 0 & 3 & 1 & 2 
    	\\
    	3 & 5 & 4 & 1 & 3 
    \end{array}\right]
     $
	
	\item $\bar{\cv}\leftarrow\left[\begin{array}{ccc}
		0 & 1 & 1 
		\\
		0 & 1 & 2 
		\\
		5 & 1 & 3 
	\end{array}\right]
	$
	
	\item $L\leftarrow \left[\begin{array}{ccc}
		0 & 1 & \frac{1}{5} 
		\\
		0 & 1 & \frac{2}{5} 
		\\
		5 & 1 & \frac{3}{5} 
	\end{array}\right]
	$
	
	\item $k\leftarrow 2$

	\item $\bar{\vv}\leftarrow \left[\begin{array}{c}
		t  
		\\
		t^{3}+t^{2} 
		\\
		5 t^{4}+5 t^{2}+5 
	\end{array}\right]
	$
	
	\item $s\leftarrow \frac{1}{3}$
\end{enumerate}
\end{example}

\begin{algorithm}[Minimal-degree Matrix Completion]\label{alg-mmc}
        $\pM\leftarrow MMC\left( \vv\right) $
        
        \begin{description}
                \item[In:] $\vv\in  \rv$
                
                \item[Out:] $\pM  \in \K[t]^{n\times n}$
        \end{description}
        
        \begin{enumerate}
                \item $\bb \leftarrow $ a minimal-degree B\'ezout vector of $\vv$
                
                \item $\uu_1,\ldots,\uu_{n-1} \leftarrow $ a $\mu$-basis of $\bb$
                                
                \item $\pM' \leftarrow$ the matrix $[\vv,\uu_1,\ldots,\uu_{n-1}]$
                
                \item $\pM \leftarrow$ the matrix obtained from $\pM'$ by  dividing its last column by $|\pM'|$
        \end{enumerate}
\end{algorithm}

\begin{example} $MMC$\ 
	\begin{enumerate}
		\item[] $\vv = \left[\begin{array}{c}
		t -\frac{1}{3} 
		\\
		t^{3}-\frac{1}{27} 
		\\
		5 t^{4}+\frac{25}{3} t^{2}+\frac{325}{81} 
		\end{array}\right]
		$
		\item $\bb \leftarrow
		\left[\begin{array}{c}
			-\frac{16}{27}-\frac{5 t}{3} 
			\\
			-t  
			\\
			\frac{1}{5} 
		\end{array}\right]
		$
		
		\item $\uu_1,\uu_{2} \leftarrow \left[\begin{array}{c}
			\frac{27}{80} 
			\\
			-\frac{9}{16} 
			\\
			1 
		\end{array}\right]
		, 
		\left[\begin{array}{c}
			-\frac{3 t}{5} 
			\\
			t +\frac{16}{45} 
			\\
			0 
		\end{array}\right]
		$
		
		\item $\pM' \leftarrow\left[\begin{array}{ccc}
			t -\frac{1}{3} & \frac{27}{80} & -\frac{3 t}{5} 
			\\
			t^{3}-\frac{1}{27} & -\frac{9}{16} & t +\frac{16}{45} 
			\\
			5 t^{4}+\frac{25}{3} t^{2}+\frac{325}{81} & 1 & 0 
		\end{array}\right]
		$
		
	  \item $\pM \leftarrow	\left[\begin{array}{ccc}
			t -\frac{1}{3} & \frac{27}{80} & -t  
			\\
			t^{3}-\frac{1}{27} & -\frac{9}{16} & \frac{5 t}{3}+\frac{16}{27} 
			\\
			5 t^{4}+\frac{25}{3} t^{2}+\frac{325}{81} & 1 & 0 
		\end{array}\right]
		$
		
	\end{enumerate}
\end{example}


\section{Appendix:\\ Minimal {\bez}vector and  $\mu$-basis constructions}\label{sect-bez-mu}
This appendix provides a theoretical background  for the minimal-degree B\'ezout vector and  $\mu$-basis algorithms that we used in our implementation.  It draws the material from 
 a published paper  \cite{hhk2017}, an unpublished manuscript \cite{omf-2017} and a Ph.D.~Thesis \cite{hough-2018}.
Theorems~\ref{thm-bez} and~\ref{thm-mu} show that, given a vector $\vv\in\mathbb{K}[t]^{n}_{d}$, 
such that $\gcd(\vv)=1$, one can deduce a  minimal-degree {B\'ezout} vector of $\vv$
and a $\mu$-basis of $\vv$ from the linear relationships
among certain columns of the same ${(2d+1)\times(nd+n+1)}$ Sylvester-type matrix over $\K$. This leads to  easy-to-implement algorithms, based on the row echelon reduction over $\K$,  whose details  are presented in the above references.

\subsection{Sylvester-type matrix $A$ and its properties}

\label{ssect-A}
For a nonzero polynomial vector 
\begin{equation}
\label{eq-a}
\vv=\underset{\cv}{\underbrace{\left[
\begin{array}
[c]{ccc}%
v_{10} & \dots & v_{1d}\\
\vdots & \vdots & \vdots\\
v_{n0} & \dots & v_{nd}
\end{array}
\right]  }}\left[
\begin{array}
[c]{c}%
t^{0}\\
\vdots\\
t^{d}%
\end{array}
\right]
\in\pv_d,%
\end{equation}
we correspond a $\mathbb{K}^{(2d+1)\times
n(d+1)}$ matrix
\beq
\label{eq-A}A=\left[
\begin{array}
[c]{cccccccccc}%
v_{10} & \cdots & v_{n0} &  &  &  &  &  &  & \\
\vdots & \cdots & \vdots & v_{10} & \cdots & v_{n0} &  &  &  & \\
\vdots & \cdots & \vdots & \vdots & \cdots & \vdots & \ddots &  &  & \\
v_{1d} & \cdots & v_{nd} & \vdots & \cdots & \vdots & \ddots & v_{10} & \cdots
& v_{n0}\\
&  &  & v_{1d} & \cdots & v_{nd} & \ddots & \vdots & \cdots & \vdots\\
&  &  &  &  &  & \ddots & \vdots & \cdots & \vdots\\
&  &  &  &  &  &  & v_{1d} & \cdots & v_{nd}%
\end{array}
\right]
\eeq
with the blank spaces filled by zeros. In other words, matrix $A$ is obtained
by taking $d+1$ copies of a $(d+1)\times n$ matrix $\cv^T$. The blocks are repeated horizontally from left to
right, and each block is shifted down by one relative to the previous one.
\begin{example}\label{ex-A} \rm
For 
$$\vv=\left[
\begin{array}
[c]{c}%
2+t+t^{4}\\
3+t^2+t^{4}\\
6+2t^3+t^{4}%
\end{array}
\right]= \left[
\begin{array}
[c]{ccccc}%
2 &1&0 &0 & 1\\
3 &0&1 &0 & 1\\
6 &0&0 &2 & 1\\
\end{array}
\right] \left[
\begin{array}
[c]{c}%
t^{0}\\
\vdots\\
t^{4}%
\end{array}
\right],
$$
we have  $n=3$, $d=4$, and a $9\times15$ matrix of Sylvester-type:
\[
A=\left[
\begin{array}
[c]{rrrrrrrrrrrrrrr}%
2 & 3 & 6 &  &  &  &  &  &  &  &  &  &  &  & \\
1 & 0 & 0 & 2 & 3 & 6 &  &  &  &  &  &  &  &  & \\
0 & 1 & 0 & 1 & 0 & 0 & 2 & 3 & 6 &  &  &  &  &  & \\
0 & 0 & 2 & 0 & 1 & 0 & 1 & 0 & 0 & 2 & 3 & 6 &  &  & \\
1 & 1 & 1 & 0 & 0 & 2 & 0 & 1 & 0 & 1 & 0 & 0 & 2 & 3 & 6\\
&  &  & 1 & 1 & 1 & 0 & 0 & 2 & 0 & 1 & 0 & 1 & 0 & 0\\
&  &  &  &  &  & 1 & 1 & 1 & 0 & 0 & 2 & 0 & 1 & 0\\
&  &  &  &  &  &  &  &  & 1 & 1 & 1 & 0 & 0 & 2\\
&  &  &  &  &  &  &  &  &  &  &  & 1 & 1 & 1
\end{array}
\right].
\]
\qq \end{example}
\begin{definition}
A column of any matrix $N$ is called \emph{pivotal} if it is either the first
column and is nonzero or it is linearly independent of all previous columns.
The rest of the columns of $N$ are called \emph{non-pivotal}. The index of a
pivotal (non-pivotal) column is called a \emph{pivotal (non-pivotal) index}.
\end{definition}

From this definition, it follows that every non-pivotal column can be written
as a linear combination of the preceding \emph{pivotal columns}.
We denote the set of pivotal indices of $A$ as $p$ and the set of its
non-pivotal indices as $q$. The following two lemmas, proved in \cite[Lemmas 17 and 19]{hhk2017} show how the specific structure of the matrix $A$ is reflected
in the structure of the set of non-pivotal indices~$q$.

\begin{lemma}
[periodicity]\label{periodic}If $j\in q$ then $j+kn\in q$ for $0\leq
k\leq\left\lfloor \frac{n(d+1)-j}{n}\right\rfloor $. Moreover,
\begin{equation}
\label{eq-periodic}A_{*j}=\sum_{ r<j} \alpha_{r}\, A_{*r}\quad\Longrightarrow
\quad A_{*j+kn}=\sum_{ r<j} \alpha_{r}\, A_{*r+kn},
\end{equation}
where $A_{*j}$ denotes the $j$-th column of $A$.
\end{lemma}


\begin{definition}
\label{def-bnp} Let $q$ be the set of non-pivotal indices. Let $q/(n)$ denote
the set of equivalence classes of $q \text{ modulo } n$. Then the set $\tilde
q=\{ \min\varrho\,| \varrho\in q/(n)\}$ will be called the set of \emph{basic
non-pivotal indices}. The remaining indices in $q$ will be called \emph{periodic non-pivotal indices}.
\end{definition}



\begin{example}
\textrm{\label{ex-period} For the matrix $A$ in Example~\ref{ex-A}, we have
$n=3$ and $q=\{ 8,9,11,12,14,15\}$.  Then $q/(n)=\big\{ \{8,11,14\},\,\{9,12,15\} \}\big\}$ and $\tilde q=\{8,9\}$. }
\qq \end{example}


\begin{lemma}
\label{lem-card} There are exactly $n-1$ basic non-pivotal indices: $|\tilde
q|=n-1$.
\end{lemma}
The following useful proposition about the  rank of matrix $A$  can be
deduced from  the results about the rank of a different Sylvester-type matrix, $R$, given in Section 2 of
\cite{vard-1978}. For a direct  proof see \cite[Lemma 27 and Proposition 28]{omf-2017}.

\begin{proposition}[full rank]
\label{lem-rank} For a nonzero polynomial vector $\vv\in \pv_d$,
defined by \eqref{eq-a}, such that $\gcd(\vv)=1$, the corresponding ${(2d+1)\times
n(d+1)}$ 
matrix $A$, defined by \eqref{eq-A}, has rank $2d+1$.
\end{proposition}

\subsection{Isomorphism between $\mathbb{K}[t]_{\dg}^{\len}$ and $\mathbb{K}%
^{\len(\dg+1)}$} \label{ssect-iso}
Our second ingredient is an explicit isomorphism  between the  vector space
$\mathbb{K}[t]_{\dg}^{\len}$ of polynomial vectors of length $\len$ and degree $\dg$  and the vector space $\mathbb{K}^{\len(\dg+1)}$ of constant vectors of length $\len(\dg+1)$.  

For fixed $\len$ and $\dg$, we define  the  \emph{sharp map} $\sharp^{\len}_{\dg}\colon\mathbb{K}[t]_{\dg}^{\len}\to\K^{\len(\dg+1)}$ by sending a polynomial vector
\begin{equation}
\label{eq-h}
\hh=\underset{\ch}{\underbrace{\left[
\begin{array}
[c]{ccc}%
h_{10} & \dots & h_{1\dg}\\
\vdots & \vdots & \vdots\\
h_{\len 0} & \dots & h_{\len\dg}
\end{array}
\right]  }}\left[
\begin{array}
[c]{c}%
t^{0}\\
\vdots\\
t^{\dg}%
\end{array}
\right]
\in\K[t]^\len_\dg,%
\end{equation}
to a constant vector  
\begin{equation}
\label{iso1}\hh^{\sharp^{\len}_{\dg}}= \left[
\begin{array}
[c]{c}%
H_{*0}\\
\vdots\\
H_{*d}%
\end{array}
\right]\in \mathbb{K}^{\len(\dg+1)}
\end{equation}
obtained by stacking the $\dg+1$ columns of the coefficient matrix $\ch$.

For fixed $\len$ and $\dg$, we define  the \emph{flat map}
\[
\flat^{\len}_{\dg}\colon\mathbb{K}^{\len(\dg+1)} \to\K[t]^\len_\dg%
\]
by
\begin{equation}
\label{iso2}h\to h^{\flat^{\len}_{\dg}} =S^{\len}_{\dg}\, h
\end{equation}
where
\[
S^{\len}_{\dg}=\left[
\begin{array}
[c]{ccc}%
I_{\len} & tI_{\len} & \cdots t^{\dg}I_{\len}%
\end{array}
\right]  \in\mathbb{K}[t]^{\len\times \len(\dg+1)} 
\]
 with  $I_{\len}$ denoting the $\len\times \len$ identity matrix.

 It is easy to check that the \emph{flat and sharp maps are linear and they are the inverse of each other}.
 {For the sake of
notational simplicity, we will often write $\sharp$, $\flat$ and $S$ instead
of $\sharp^{\len}_{\dg}$, $\flat^{\len}_{\dg}$ and~$S^{\len}_{\dg}$ when the values of $\len$
and $\dg$ are clear from the context.}


\begin{example}
\label{ex-sharp-flat}\textrm{For $\mathbf{h}\in\mathbb{Q}^{3}_{3}[s]$ given
by
\[
\mathbf{h}=\left[
\begin{array}
[c]{c}%
9-12t-t^{2}\\
8+15t\\
-7-5t+t^2
\end{array}
\right] = \left[
\begin{array}
[c]{rrr}%
9&-12&-1\\
8&15&0\\
-7&-5&1
\end{array}
\right]\left[
\begin{array}
[c]{c}%
1\\
t\\
t^2
\end{array}
\right]  
\]
we have
\[
\mathbf{h}^{\sharp}=[9,\,8,\,-7,\,-12,\,15,\,-5,\,-1,\,0,\,1]^{T}.
\]
Note that
\[
\mathbf{h}=(\mathbf{h}^{\sharp})^{\flat}=S\, \mathbf{h}^{\sharp}=\left[
\begin{array}
[c]{ccc}%
I_{3} & tI_{3} & t^{2}I_{3}
\end{array}
\right]  \mathbf{h}^{\sharp}.
\]
}
\qq \end{example}

The following lemma shows that, with  respect to the isomorphisms $\sharp$ and $\flat$, the $\mathbb{K}$-linear
map $\mathbb{K}[t]_{d}^{n}\to\mathbb{K}[t]_{2d}$ defined by the scalar product with a fixed vector $\vv\in \pv_d$ corresponds 
to the $\mathbb{K}$ linear map $A\colon\mathbb{K}^{n(d+1)}\to\mathbb{K}%
^{2d+1}$ in the following sense:
\begin{lemma}\label{lem-aA}
 For a nonzero polynomial vector $\vv\in\pv_d$,
defined by \eqref{eq-a}, the corresponding ${(2d+1)\times
n(d+1)}$ 
matrix $A$, defined by \eqref{eq-A},  a constant vector $h\in\K^{n(d+1)}$ and a polynomial vector $\hh\in\pv_d$, we have:
\begin{equation}
\label{eq-aA}\big<\vv,h^{\flat^{n}_{d}}\big>=(A h)^{\flat^{1}_{2d}} \text{ and } \big<\vv, \hh\big>^{\sharp^{1}_{2d}}=A \left(\hh^{\sharp^{n}_{d}}\right)
\end{equation}
\end{lemma}

The proof of Lemma~\ref{lem-aA} is straightforward. The proof of the first
equality is explicitly spelled out in \cite{hhk2017} (see Lemma 10). The
second equality follows from the first and the fact that $\flat^{m}_{t}$
and~$\sharp^{m}_{t}$ are mutually inverse maps.

\begin{example}
\label{ex-Aa}\rm Consider  $\vv$  and its associated matrix $A$ in Example \ref{ex-A}.
Let $$h=[1,2,3,4,5,6,7,8,9,10,11,12,13,14,15]^{T}.$$ Then
\[
Ah= [26,60,98,143,194,57,62,63,42]^{T}%
\]
and so
\[
(Ah)^{\flat_{2d}^{1}} = S_{8}^{1}(Ah) = 26 + 60t + 98t^{2} + 143t^{3} + 194t^{4}
+ 57t^{5} + 62t^{6} + 63t^{7} + 42t^{8}.
\]
We invite the reader to check that since
\[
h^{\flat^{n}_{d}} = S_{4}^{3}h = \left[
\begin{array}
[c]{c}%
1+4t+7t^{2}+10t^{3}+13t^{4}\\
2+5t+8t^{2}+11t^{3}+14t^{4}\\
3+6t+9t^{2}+12t^{3}+15t^{4}%
\end{array}
\right],
\]
we have
\[
\big<\vv,h^{\flat^n_d}\big> =(Ah)^{\flat^1_{2d}}.
\]
\qq \end{example}

\subsection{The minimal B\'ezout vector theorem}

\label{ssect-bez} In this section, we construct a minimal-degree B\'ezout vector of $\vv$ by finding an appropriate solution to the linear equation
\begin{equation}
\label{eq-bezA}A\,b=e_{1}, \text{ where } e_{1}=[1,0,\,\dots,\, 0]^{T}%
\in\mathbb{K}^{2d+1}.
\end{equation}
The following lemma establishes a one-to-one correspondence between the set
$\operatorname{Bez}_{d}(\vv)$ of B\'ezout vectors of $\vv$ of degree \emph{at most} $d$ and the set
of solutions to \eqref{eq-bezA}.

\begin{lemma}
\label{lem-bez} Let $\vv\in\mathbb{K}[t]_{d}^{n}$ be a nonzero vector
such that $\gcd(\vv)=1$. Then $\bb\in\mathbb{K}[t]^{n}_{d}$
belongs to $\operatorname{Bez}_{d}(\vv)$ if and only if $\bb^{\sharp}$ is a solution of \eqref{eq-bezA}. Also $b\in\mathbb{K}^{n(d+1)}$
is a solution of \eqref{eq-bezA} if and only if $b^{\flat}$ belongs to
$\operatorname{Bez}_{d}(\vv)$.

\end{lemma}

\begin{proof}
Follows immediately from \eqref{eq-aA} and the observation that $e_{1}%
^{\flat^{1}_{2d}}=1$.
\end{proof}

\medskip

Now, the goal is to construct a solution $b$  of  \eqref{eq-bezA}, such that
$b^{\flat}$ is a {B\'ezout} vector of $\vv$ of minimal degree. To accomplish this,
we recall that, when $\gcd(\vv)=1$, Proposition~\ref{lem-rank} asserts
that $\mathrm{rank} (A)=2d+1$. Therefore, $A$ has exactly $2d+1$ pivotal
indices, which we can list in the increasing order $p=\{p_{1},\dots,
p_{2d+1}\}$. The corresponding columns of matrix $A$ form a basis of~$\mathbb{K}^{2d+1}$ and, therefore, $e_{1}\in\mathbb{K}^{2d+1}$ can be
expressed as a unique linear combination of the pivotal columns:
\begin{equation}
\label{eq-e1comb}e_{1}=\sum_{j=1}^{2d+1}\alpha_{j}A_{*p_{j}}.
\end{equation}
Define vector $b\in\mathbb{K}^{n(d+1)}$ by setting its $p_{j}$-th element to be
$\alpha_{j}$ and all other elements to be 0. We prove that $\bb=b^{\flat}$ is a {B\'ezout} vector of $\vv$ of the minimal degree.
\begin{theorem}
[Minimal-Degree {B\'ezout} Vector]\label{thm-bez} Let $\vv\in\mathbb{K}%
[t]^{n}_{d}$ be a polynomial vector with $\gcd(\vv)=1$, and let $A$ be the
corresponding matrix defined by \eqref{eq-A}. Let $p=\{p_{1},\dots,
p_{2d+1}\}$ be the pivotal indices of $A$, and let $\alpha_{1},\dots
,\alpha_{2d+1}\in\mathbb{K}$ be defined by the unique expression
\eqref{eq-e1comb} of the vector $e_{1}\in\mathbb{K}^{2d+1}$ as a linear
combination of the pivotal columns of $A$. Define vector~$b\in\mathbb{K}%
^{2d+1}$ by setting its $p_{j}$-th element to be $\alpha_{j}$ for $j=1,\dots,
2d+1$ and all other elements to be 0, and let~$\mathbf{b}=b^{\flat}$. Then

\begin{enumerate}
\item $\bb\in\operatorname{Bez}_d(\vv)$

\item $\deg(\bb)=\displaystyle{\min_{\bb^{\prime}\in
\operatorname{Bez}(\vv)}} \deg(\mathbf{\bb}^{\prime}).$ 
\end{enumerate}
\end{theorem}
\begin{proof}\hfill
\begin{enumerate}
\item From \eqref{eq-e1comb}, it follows immediately that $Ab=e_{1}$. 
Therefore, by Lemma~\ref{lem-bez}, we have that~$\bb=b^{\flat}%
\in\operatorname{Bez}_{d}(\vv)$.

\item To show that $\mathbf{b}$ is of minimal degree, we rewrite
\eqref{eq-e1comb} as
\begin{equation}
\label{eq-e1comb1}e_{1}=\sum_{j=1}^{k}\alpha_{j} A_{*p_{j}},
\end{equation}
where $k$ is the largest integer between 1 and $2d+1$, such that $\alpha
_{k}\neq0$. Then the last nonzero entry of $b$ appears in $p_{k}$-th position
and, therefore,
\begin{equation}
\label{degw}\deg(\mathbf{b})=\deg(b^{\flat})=\big\lceil{p_{k}/n}\big\rceil-1.
\end{equation}
Assume that $\mathbf{b}^{\prime}\in\operatorname{Bez}(\vv)$ is such
that $\deg(\bb^{\prime})<\deg(\bb)$. Then $\mathbf{b}^{\prime
}\in\operatorname{Bez}_{d}(\vv)$ and, therefore, $A\,b^{\prime}=e_{1}$,
for $b^{\prime}=\mathbf{b}^{\prime\sharp}=[b^{\prime}_{1},\dots, b^{\prime
}_{n(d+1)}]\in\mathbb{K}^{n(d+1)}$. Then
\begin{equation}
\label{eq-e1v1}e_{1}=\sum_{j=1}^{n(d+1)}\,b^{\prime}_{j}A_{*j} =\sum_{j=1}%
^{r}b^{\prime}_{j} A_{*j},
\end{equation}
where $r$ is the largest integer between 1 and $n(d+1)$, such that $b^{\prime
}_{r}\neq0$.  Then
\begin{equation}
\label{degw'}\deg(\bb^{\prime})=\big\lceil{r/n}\big\rceil-1
\end{equation}
and since we assumed that $\deg(\bb^{\prime})<\deg(\bb)$, we
conclude from \eqref{degw} and \eqref{degw'} that $r<p_{k}$.

On the other hand, since all non-pivotal columns are linear combinations of
the preceding pivotal columns, we can rewrite \eqref{eq-e1v1} as
\begin{equation}
\label{eq-e1v2}e_{1}=\sum_{j\in\{1,\dots, 2d\,|\, p_{j}\leq r<p_{k}\}}%
\alpha^{\prime}_{j} A_{*p_{j}}=\sum_{j=1}^{k-1}\alpha^{\prime}_{j} A_{*p_{j}}.
\end{equation}
By the uniqueness of the representation of $e_1$ as a linear combination of the $A_{\ast p_j}$, the coefficients in the expansions \eqref{eq-e1comb1} and
\eqref{eq-e1v2} must be equal, but $\alpha_{k}\neq0$ in \eqref{eq-e1comb1}. Contradiction.
\end{enumerate}
\end{proof}

The algorithm presented in  \cite{omf-2017,hough-2018}  exploits the fact that
the coefficients $\alpha$'s in \eqref{eq-e1comb1}, needed to construct the minimal-degree
{B\'ezout} vector of $\vv$ prescribed by Theorem~\ref{thm-bez}, can be read off the reduced row-echelon form $[\hat A| \hat
b]$ of the augmented matrix $[A|e_{1}]$.

\subsection{The $\mu$-bases theorem}
\label{ssect-mu}
In \cite{hhk2017}, it is  shown that the coefficients of a $\mu$-basis of $\vv\in\pv_d$ can be read off the \emph{basic}
non-pivotal columns of matrix $A$ (recall Definition~\ref{def-bnp}). According to Lemma~\ref{lem-card}, the matrix $A$ has exactly $n-1$ basic
non-pivotal columns.

\begin{theorem}
[$\mu$-Basis]\label{thm-mu} Let $\vv\in\mathbb{K}[t]^{n}_{d}$ be a
polynomial vector, and let $A$ be the corresponding matrix defined by \eqref{eq-A}.
Let $\tilde q=[\tilde q_{1},\dots, \tilde q_{n-1}]$ be the basic non-pivotal
indices of $A$, ordered increasingly. For $i=1,\dots,{n-1}$, a basic non-pivotal
column $A_{*\tilde q_{i}}$ is a linear combination of the previous pivotal
columns:
\begin{equation}
\label{eqAij}A_{*\tilde q_{i}}=\sum_{\{r\in p\,|\, r<\tilde q_{i}\}}
\alpha_{ir} A_{*r},
\end{equation}
for some $\alpha_{ir}\in\mathbb{K}$. Define vector $u_{i}\in\mathbb{K}^{n(d+1)}$
by setting its $\tilde q_{i}$-th element to be $1$, its $r$-th element to be
$-\alpha_{ir}$ for $r\in p$ such that $p_{j}<\tilde q_{i}$, and all other
elements to be 0. Then the set of polynomial vectors
\[
\mathbf{u}_{1}=u_{1}^{\flat},\quad\dots\quad, \mathbf{u}_{n-1}=u_{n-1}^{\flat}%
\]
is a degree-ordered $\mu$-basis of $\vv$.
\end{theorem}
\begin{proof}
The fact that $\mathbf{u}_{1}=u_{1}^{\flat},\quad\dots\quad, \mathbf{u}%
_{n-1}=u_{n-1}^{\flat}$ is a $\mu$-basis of $\vv$ is the statement of Theorem~27 of
\cite{hhk2017}. By construction, the last nonzero entry of vector $u_{i} $ is
in the $\tilde q_{i}$-th position, and therefore for $i=1,\dots, n-1$,
\[
\deg(\mathbf{u}_{i}) =\deg(u_{i}^{\flat})=\big\lceil{\tilde q_{i}%
/n}\big\rceil-1.
\]
Since the indices in $\tilde q$ are ordered increasingly, the vectors
$\mathbf{u}_{1},\, \dots,\, \mathbf{u}_{n-1}$ are degree-ordered.
\end{proof}


The algorithm presented in \cite{hhk2017}   exploits the fact that
the coefficients $\alpha$'s in \eqref{eqAij}, needed to construct the $\mu$-basis of $\vv$ prescribed by Theorem~\ref{thm-mu},
can be read off the reduced row-echelon form $\hat A$ of $A$.

\bibliographystyle{acm}

\end{document}